\documentclass{amsart}
\usepackage{amsmath, amssymb, epsfig, verbatim, graphs, psfrag}

\newtheorem{thm}{Theorem}[section]

\newtheorem{cor}[thm]{Corollary}
\newtheorem{lem}[thm]{Lemma}
\newtheorem{prop}[thm]{Proposition}
\newtheorem{defn}[thm]{Definition}

\newtheorem{exa}[thm]{Example}
\newtheorem{exas}[thm]{Examples}
\newtheorem{rem}[thm]{Remark}
\newtheorem{rems}[thm]{Remarks}

\numberwithin{equation}{section}

\newcommand\Field{\mathbb F}

\newcommand\Dual{\mathcal D}
\newcommand\Duality\Dual

\newcommand\ModSphere{\ModFlow\left({\mathbb S}\longrightarrow 
\Sym^{g-1}(\Sigma_{1})\times \Sym^2(\Sigma_{2})\right)}
\newcommand\ModSpheres\ModSphere

\newcommand\HFm{\HF^-}

\newcommand\UnparModSp{\widehat \ModSp}
\newcommand\UnparModFlow\UnparModSp
\newcommand\Mod\ModSp

\newcommand\ModMaps{\mathcal M}
\newcommand\ModSp\ModMaps

\newcommand{\bfz}{{\mathbb {Z}}}

\newcommand{\Zmod}[1]{{\mathbb Z}/{#1}{\mathbb Z}}

\newcommand{\uuinv}{U^{-1},U}
\newcommand{\Sym}{{\mathrm {Sym}}}
\newcommand{\s}{\mathbf s}   \renewcommand{\t}{\mathbf t}

 \newcommand{\Z}{\mathbb Z}   \newcommand{\R}{\mathbb R}

\newcommand{\grad}{{\rm {gr}}}

\newcommand\Vertices{\mathrm{Vert}}

\DeclareMathOperator{\Id}{Id}

\begin{document}

\title{Knots in lattice homology}

\author{Peter Ozsv\'ath}
\address{Department of Mathematics, Princeton University\\
Princeton, NJ, 08544}
\email{petero@math.princeton.edu}

\author{Andr\'{a}s I. Stipsicz}
\address{R{\'e}nyi Institute of Mathematics\\
Budapest, Hungary \\
and Institute for Advanced Study, Princeton, NJ, 08540}
\email{stipsicz@math-inst.hu}

\author{Zolt\'an Szab\'o}
\address{Department of Mathematics, Princeton University\\
Princeton, NJ, 08544}
\email{szabo@math.princeton.edu}


\begin{abstract}
  Assume that $\Gamma _{v_0}$ is a tree with vertex set $\Vertices
  (\Gamma _{v_0})=\{ v_0, v_1, \ldots , v_n \}$, and with an integral framing
  (weight) attached to each vertex except $v_0$. Assume furthermore that the
  intersection matrix of $G=\Gamma _{v_0}-\{ v_0\}$ is negative
  definite. We define a filtration on the chain
  complex computing the lattice homology of $G$ and show how to use
  this information in computing lattice homology groups of a negative
  definite graph we get by attaching some framing to $v_0$. As a
  simple application we produce families of graphs which have
  arbitrarily many bad vertices for which the lattice homology groups
  are shown to be isomorphic to the corresponding Heegaard Floer
  homology groups.
\end{abstract}
\maketitle

\newcommand\CFinfComb{{\mathbb{CF}}^{\infty}}
\newcommand\HFinfComb{{\mathbb{HF}}^{\infty}}
\newcommand\CFmComb{{\mathbb{CF}}^{-}}
\newcommand\CFpComb{{\mathbb{CF}}^{+}}
\newcommand\HFmComb{{\mathbb{HF}}^{-}}
\newcommand\HF{\rm {HF}}
\newcommand\HFpComb{{\mathbb{HF}}^{+}}
\newcommand\HFKmComb{{\mathbb{HFK}}^{-}}
\newcommand\HFKinfComb{{\mathbb{HFK}}^{\infty}}
\newcommand\MCFinfComb{{\mathbb{MCF}}^{\infty}}
\newcommand\CFaComb{\widehat {\mathbb{CF}}}
\newcommand\CFoverComb{\overline {\mathbb{CF}}}
\newcommand\HFaComb{\widehat {\mathbb{HF}}}
\newcommand\HFoverComb{\overline {\mathbb{HF}}}
\newcommand\HFKaComb{\widehat {\mathbb{HFK}}}
\newcommand\Char{\mathrm{Char}}

\section{Introduction}
It is an eminent problem in low dimensional topology to find simple
computational schemes for the recently defined invariants
(e.g. Heegaard Floer and Monopole Floer homologies) of 3- and
4-manifolds. In particular, the minus-version $\HF ^-$ of Heegaard
Floer homology (defined over the polynomial ring $\Field [U]$, where
$\Field$ denotes either $\Z$ or the field $\Z /2\Z$ of two elements)
is of central importance. In \cite{MOT} a computational scheme for
the $\HF ^-$ groups was presented, which is rather hard to implement
in practice. This result was preceded by a more practical way
of determining these invariants for those 3-manifolds which can be
presented as boundary of a plumbing of spheres along a negative
definite tree which has at most one bad vertex \cite{OSzplum}. The
idea of \cite{OSzplum} was subsequently extended by N\'emethi
\cite{nemethi-ar}, and in \cite{lattice} a new invariant,
\emph{lattice homology} was proposed.  It has been conjectured that
lattice homology determines the Heegaard Floer groups when the
underlying 3-manifold is given by a negative definite plumbing of
spheres along a tree. Common features (eg. the existence of surgery
exact triangles) have been verified for the two theories (in
\cite{OSzF1} for the Heegaard Floer setting, while in \cite{Josh,
  latticetriangle} for lattice homology), and the existence of a
spectral sequence connecting the two theories has been found
\cite{OSSZlatt}. For further related results see
\cite{properties-lattice, s3csomok}.

In the present work we extend these similarities by introducing
filtrations on lattice homologies induced by vertices, mimicking the
ideas of knot Floer homologies developed in the Heegaard Floer context
in \cite{OSzknot, Ras}.  This information then (just as in the Heegaard
Floer context) can be conveniently used to determine the lattice
homology of the graph when the distinguished vertex is equipped with
some framing (corresponding to the surgery formulae in Heegaard Floer
theory, cf. \cite{OSzint}).

In more concrete terms, suppose that $\Gamma _{v_0}$ is a given tree
(or forest), with each vertex $v$ in $\Vertices (\Gamma _{v_0})-v_0$
equipped with a framing (or weight) $m_v\in \Z$. Let $G$ denote the tree (or
forest) we get by deleting $v_0$ and the edges emanating from it.
Suppose that $G$ is negative definite. We will define the \emph{master
  complex} $\MCFinfComb (\Gamma _{v_0})$ of $\Gamma _{v_0}$, which is a 
filtration on the chain complex defining the lattice homology of $G$
together with a specific map, and will show
\begin{thm}
The master complex $\MCFinfComb  (\Gamma _{v_0})$  determines the lattice
  homology of all negative definite framed trees (or forests) we get
  from $\Gamma _{v_0}$ by attaching framings to $v_0$.
\end{thm}

By identifying the filtered chain homotopy type of the resulting
master complex with the knot Floer homology of the
corresponding knot in the plumbed 3-manifold, this method allows us to
show that certain graphs have identical lattice and Heegaard Floer
homologies.  A connected sum formula then enables us to extend this
method to further graphs, including some with arbitrarily many bad
vertices.  As an example, we show
\begin{thm}\label{thm:pelda}
Consider the plumbing graph of Figure~\ref{fig:connsum} on
$3n+1$ vertices, with the framing of $v_0$ an integer at most 
$-6n-1$. Then the lattice homology
of the graph is isomorphic to the Heegaard Floer
homology $\HFm$ of the 3-manifold defined by the plumbing.
\end{thm}
\begin{rem}
  Notice that the graph of Figure~\ref{fig:connsum} on $3n+1$ vertices
  (after we attach a framing $-m\leq -6n-1$ to the central vertex
  $v_0$) has $n$ bad vertices. The case of $n=2$ in the theorem was
  already proved by N\'emethi, cf. \cite[Example~4.4.1]{lattice}, see
  also \cite{s3csomok} for related results. For a more general result
  along similar lines, see \cite{OSSzlspace}.
\end{rem}

\begin{figure}[ht]
\begin{center}
\epsfig{file=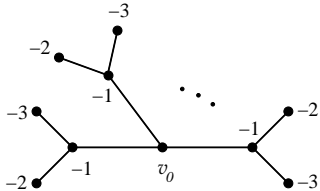, height=2.5cm}
\end{center}
\caption{{\bf The plumbing diagram of the $n$-fold connected sum of
    the (right-handed) trefoil knot in $S^3$.} The valency of the
  central vertex $v_0$ is assumed to be $n\in {\mathbb {N}}$, and each
  edge emanating from $v_0$ connects it to a vertex with framing
  $(-1)$. Furthermore these $(-1)$-vertices are connected to a $(-2)$-
  and a $(-3)$-framed leaf of the graph. Regarding $v_0$ as a circle
  in the plumbed 3-manifold defined by the rest of the graph, it can
  be identified with the $n$-fold connected sum of the trefoil knot in
  $S^3$.}
\label{fig:connsum}
\end{figure}

As an application of the connected sum formula, in an Appendix we give
an alternative proof of the following result of N\'emethi.
\begin{thm} \cite[Proposition~3.4.2]{lattice} 
\label{thm:indep}
Suppose that the two negative definite plumbing trees (or forests)
$G_1$ and $G_2$ define diffeomorphic 3-manifolds $Y_{G_1}$ and
$Y_{G_2}$.  Then the lattice homology $\HFmComb (G_1)$ of $G_1$ is
isomorphic to the lattice homology $\HFmComb (G_2)$ of $G_2$.  In
other words, the lattice homology is an invariant of the 3-manifold
defined by the plumbing graph.
 \end{thm}

 The paper is organized as follows. In Section~\ref{sec:lattice} we
 review the basics of lattice homology for negative definte graphs. In
 Sections~\ref{sec:knots} and \ref{sec:master} we introduce the knot
 filtration on the lattice chain complex of the background graph,
 describe the master complex and verify the connected sum formula.  In
 Section~\ref{sec:surgery} we show how to apply this information to
 determine the lattice homology of graphs we get by attaching various
 framings to the distinguished point $v_0$.  In
 Section~\ref{sec:trefoil} we determine the knot filtration in one
 specific example, and verify Theorem~\ref{thm:pelda}.  Finally in
 Section~\ref{sec:app} we give a proof of Theorem~\ref{thm:indep}
 (another proof appeared in \cite{latticetriangle}), and in
 Section~\ref{app:contract} we reprove a further result of N\'emethi
 \cite{lattice} stating that lattice homology is finitely generated as an
 $\Field [U]$-module.

\subsection*{Notation}
Suppose that $\Gamma$ is a tree (or forest), and $G$ is the same graph
equipped with framings, i.e. we attach integers to the vertices of
$\Gamma$. The plumbing of disk bundles over spheres defined by $G$
will be denoted by $X_G$, and its boundary 3-manifold is $Y_G$.  Let
$M_G$ denote the incidence matrix associated to $G$ (with framings in
the diagonal). This matrix presents the intersection form of $X_G$ in
the basis provided by the vertices of the plumbing graph.

Suppose that $\Gamma _{v_0}$ is a plumbing tree (or forest) with a
distinguished vertex $v_0$ which is left unframed (but all
other vertices of $\Gamma _{v_0}$ are framed). Let $G$ denote the plumbing
graph we get by deleting the vertex $v_0$ (and all the edges adjacent
to it). We will always assume that the plumbing trees/forests we work
with are negative definite.

\begin{rem}
We can regard the unknot defined by $v_0$ in the plumbing picture as a 
(not necessarily trivial) knot in the plumbed 3-manifold $Y_G$.
\end{rem}

Recall that for a negative definite tree (or forest) $G$ on the vertex
set $\Vertices (G)$ the vertex $v\in \Vertices (G)$ is a \emph{bad vertex}
if $m_v+d_v>0$, where $m_v$ denotes the framing attached to $v$ while
$d_v$ is the valency or degree of $v$ (the number of edges emanating
from $v$). A vertex is \emph{good} if it is not bad, that is,
$m_v+d_v\leq 0$.

{\bf Acknowledgements}: PSO was supported by NSF grant number
DMS-0804121.  AS was supported by OTKA NK81203, by the ERC Grant
LDTBud, by \emph{Lend\"ulet program} and by the Institute for Advanced
Study.  ZSz was supported by NSF grants DMS-0603940, DMS-0704053 and
DMS-1006006.

\section{Review of lattice homology}
\label{sec:lattice}

Lattice homology has been introduced by N\'emethi in \cite{lattice}
(cf. also \cite{properties-lattice, latticetriangle, s3csomok}). In
this section we review the basic notions and concepts of this theory.
Our main purpose is to set up notations which will be used in the rest
of the paper.

Following \cite{lattice}, for a given negative definite plumbing tree
$G$ we define a $\Z$-graded combinatorial chain complex $(
\CFinfComb(G), \partial )$ (and then a subcomplex $(\CFmComb
(G), \partial )$ of it), which is a module over the ring of Laurent
polynomials $\Field[\uuinv ]$ (and over the polynomial ring $\Field
[U]$, respectively), where we assume for simplicity that $\Field =
\Zmod{2}$.

Define $Char (G)$ as the set of characteristic cohomology elements
of $H^2 (X_G; \Z )$,  that is, 
\[
Char (G)=\{ K\colon H_2(X_G; \Z )\to \Z \mid K(x)\equiv x\cdot x \pmod{2}\}.
\]

The \emph{lattice chain complex} $\CFinfComb(G)$ is freely generated
over $\Field [\uuinv ]$ by the product $\Char(G)\times {\mathbb
  {P}}(\Vertices (G))$, that is, by elements $[K,E]$ where
$K\in\Char(G)\subset H^2 (X_G; \Z )$ and $E\subset\Vertices (G)$.  We
introduce a $\Z$-grading on this complex, called the {\em
  $\delta$-grading}, which is defined on the generator $[K,E]$ as the
number of elements in $E$. To define the boundary map of the chain complex, we proceed as
follows.  Given a subset $I\subset E$, we define the \emph{$G$-weight}
$f([K,I])$ as the quantity
\begin{equation}\label{eq:fdef}
2 f([K,I])= \left(\sum_{v\in I} K(v)\right) + \left(\sum_{v\in I}
  v\right) \cdot \left(\sum_{v\in I} v\right).
\end{equation}
\begin{rem}\label{rem:altfor}
  Using the fact that $G$ is negative definite, the integer $f([K,I])$
  can be easily shown to be equal to
\[
\frac{1}{8}\left( (K+\sum _{v\in I} 2v^*)^2-K^2 \right) ,
\]
where $v^*\in H^2 (X_G, Y_G; \Z )$ denotes the Poincar\'e dual of the
class $v\in H_2(X_G; \Z )$ corresponding to the vertex $v\in \Vertices
(G)$. This form of $f(K,I)$ immediately implies, for example, the
following useful identity: if $I\subset E$ then
\begin{equation}\label{eq:useful}
f([K, I])-f([-K-\sum _{u\in E}2u^*, E-I])=f([K,E]).
\end{equation}
\end{rem}

We define the \emph{minimal $G$-weight} $g([K,E])$ of $[K,E]$ by the formula
\[
g([K,E])=\min \{ f([K,I])\mid {I\subset E} \}.
\]
The quantities $A_v([K,E])$ and $B_v([K,E])$ are defined as follows:
\begin{eqnarray*}             
  A_v([K,E])=g([K,E-v]) &{\text{and}} &                                      
  B_v([K,E])=\min \{  f([K,I]) \mid {v\in I\subset E}\} . 
\end{eqnarray*}
A simple argument shows that
\begin{equation}\label{eq:refor}
B_v([K,E])=\left(\frac{K(v)+v^2}{2}\right) + g([K+2v^*,E-v]).
\end{equation}
It follows trivially from the definition that
\[
\min \{ A_v([K,E]), B_v([K,E])\} =g([K,E]).
\]
Now we define the boundary map $\partial \colon \CFinfComb (G) \to \CFinfComb (G)$ by the formula:
\[
{\partial [K,E]} = \sum_{v\in E} U^{a_v[K,E]}\otimes [K,E-v] + \sum_{v\in E}
U^{b_v[K,E]}\otimes [K+2v^*,E-v],
\]
where
\begin{eqnarray*}     a_v[K,E]=A_v([K,E])-g([K,E]) 
&{\text{and}} & b_v[K,E]=B_v([K,E])-g([K,E]).
\end{eqnarray*}
(Extend this map $U$-equivariantly to the terms $U^j \otimes [K,E]$
and then linearly to $\CFinfComb (G)$.) Notice that $a_v[K,E],
b_v[K,E]$ are both nonnegative integers and $\min \{ a_v [K,E],
b_v[K,E] \} =0$ follows directly form the definitions.  It is obvious
that the boundary map decreases the $\delta$-grading by one.  Furthermore,
it is a simple exercise to show that
\begin{lem}
  The map $\partial $ is a boundary map, that is, $\partial ^2 =0$.
\end{lem}
\begin{proof}
The proof boils down to matching the exponents of the $U$-factors in 
front of various terms in $\partial ^2 [K,E]$ for a given generator
$[K,E]$. This idea leads us to four equations to check. One of them,
for example, relates the two $U$-powers in front of the two appearances 
$[K,E-v_1-v_2]$ in $\partial ^2[K,E]$. We claim that 
\begin{equation}\label{eq:olto}
a_{v_1}[K,E] +a_{v_2}[K,E-v_1]=a_{v_2}[K,E]+a_{v_1}[K,E-v_2]
\end{equation}
holds, therefore (over $\Field$) the two terms cancel each other.
Writing out the definitions of the terms in \eqref{eq:olto} we get
\[
g([K,E-v_1])-g([K,E])+g([K,E-v_1-v_2])-g([K,E-v_1])=
\]
\[
=g([K,E-v_2])-g([K,E])+g([K,E-v_1-v_2])-g([K,E-v_2]),
\]
which trivially holds. The remaing three cases to check are:
\begin{equation}
  \label{eq:olto2}
a_{v_1}[K,E]+b_{v_2}[K,E-v_1]=b_{v_2}[K,E]+a_{v_1}[K+2v_2^*, E-v_2],
\end{equation}
\[
b_{v_1}[K,E]+a_{v_2}[K+2v_1^*,E-v_1]=a_{v_2}[K,E]+b_{v_1}[K, E-v_2],
\]
and finally
\[
b_{v_1}[K,E]+b_{v_2}[K+2v_1^*,E-v_1]=b_{v_2}[K,E]+b_{v_1}[K+2v_2^*, E-v_2].
\]
Using the definition of $B_v$ given in \eqref{eq:refor}, the equations
reduce to similar equalities as in the first case.
\end{proof}

\begin{rem}
  In \cite{lattice} the theory is set up over $\Z$; for
  simplicity in the present paper we use the coefficients from the
  field $\Field = \Z /2\Z$ of two elements.
\end{rem}

\subsection{Connected sums}
Suppose that the plumbing forest $G$ is the union of $G_1$ and $G_2$,
with no edges connecting any vertex of $G_1$ to any vertex of
$G_2$. (In other words, $G_1$ and $G_2$ are both unions of components
of $G$.) It is a simple topological fact that in this case $Y_G$
decomposes as the connected sum of the two 3-manifolds $Y_{G_1}$ and
$Y_{G_2}$.  Correspondingly, the $\Field [\uuinv ]$-module
$\CFinfComb (G)$ decomposes as the tensor product 
\begin{equation}
\label{eq:ConnectedSum}
\CFinfComb (G)\cong \CFinfComb(G_1)\otimes _{\Field [\uuinv ]} 
\CFinfComb (G_2),
\end{equation} 
and the
definition of the boundary map $\partial$ shows that this
decomposition holds on the chain complex level as well.

\subsection{Spin$^c$ structures and the $J$-map}
Define an equivalence relation for the generators of the chain complex
$\CFinfComb (G)$ as follows: we say that $[K,E]$ and $[K',E']$ are
\emph{equivalent} if $K-K'\in 2H^2(X_G, Y_G; \Z )$. Obviously, the
boundary map respects this equivalence relation, hence the chain
complex splits according to this relation.

It is easy to see that (since $X_G$ is simply connected) a
characteristic cohomology class $K\in H^2 (X_G; \Z )$ uniquely
determines a spin$^c$ structure on $X_G$. By restricting this
structure to the boundary 3-manifold $Y_G$ we conclude that $K$
naturally induces a spin$^c$ structure $\s _K$ on $Y_G$. Two classes
$K,K'$ induce the same spin$^c$ structure on $Y_G$ if and only if  they
are equivalent in the above sense (that is, $K-K'\in 2H^2 (X_G, Y_G;
\Z )$).  Therefore the splitting of the chain complex $\CFinfComb (G)$
described above is  parametrized by the spin$^c$ structures of
$Y_G$:
\[
\CFinfComb (G)=\sum _{\s \in Spin^c(Y_G)}\CFinfComb (G, \s),
\]
where $\CFinfComb (G, \s )$ is spanned by those pairs $[K,E]$ for which
$\s _K=\s$. 

Consider the map
\[
J[K,E]=[-K-\sum _{v\in E}2v^*, E]
\]
and extend it $U$-equivariantly (and linearly) to $\CFinfComb (G)$.
Obviously $J$ provides an involution on $\CFinfComb (G)$, and a simple
calculation shows the following:
\begin{lem}\label{l:spinc}
The $J$-map is a chain map, that is,
$J\circ \partial  = \partial \circ J$.
\end{lem}
\begin{proof}
The two compositions can be easily determined as
\[
(J\circ \partial )[K,E]=\sum _{v\in E} \left( U^{a_v[K,E]}\otimes
  [-K-\sum _{u\in E-v}2u^*, E-v] \right) +
\]
\[
\sum _{v\in E} \left( U^{b_v[K,E]}\otimes [-K-\sum _{u\in E}2u^*, E-v]
\right)
\]
and
\[
(\partial \circ J )[K,E]=\sum _{v\in E} \left(U^{a_v[-K-\sum _{u\in E} 2u^*,
  E]}\otimes [-K-\sum _{u\in E}2u^* , E-v]\right) +
 \]
 \[
 + \sum _{v\in E} \left(U^{b_v[-K-\sum _{u\in E}2u^*, E]}\otimes
   [-K-\sum _{u\in E-v}2u^*, E-v]\right) .
\]
The fact that $J$ is a chain map, then follows from the
two identities
\begin{equation}\label{eq:spinek}
a_v[K,E] = b_v[-K-\sum _{u\in E} 2u^*, E] \qquad
{\mbox {and}} \qquad a_v[-K-\sum _{u\in E} 2u^*,E]=b_v[K,E].
\end{equation}
In turn, these identities easily follow from the identity of~ \eqref{eq:useful},
concluding the proof of the lemma.
\end{proof}

The $J$-map obviously respects the splitting of $\CFinfComb (G)$
according to spin$^c$ structures. In fact, the spin$^c$ structures
represented by $K$ and $-K$ are 'conjugate' to each other as spin$^c$
structures on $Y_G$ (cf. \cite{OSzF1}), inducing the spin$^c$
structures $\s, {\overline {\s}}\in$Spin$^c(Y_G)$, respectively.  The
$J$-map therefore is just the manifestation of the conjugation
involution of spin$^c$ structures on the chain complex level.  Indeed,
$J$ provides an isomorphism between the two subcomplexes $\CFinfComb
(G, \s )$ and $\CFinfComb (G, {\overline {\s}})$.

\subsection{Gradings}
The lattice chain complex $\CFinfComb (G)$ admits a \emph{Maslov
  grading}: for 
a generator $[K,E]$ and $j\in {\mathbb {Z}}$ define
$\grad (U^j \otimes [K,E])$ by the formula:
\[
\grad (U^j\otimes [K,E])=-2j +2g([K,E])+\vert E\vert +\frac{1}{4}(K^2
+\vert \Vertices (G)\vert ).
\]
Recall that $K^2$ is defined as the square of $nK$ divided by $n^2$,
where $nK\in H^2 (X_G, Y_G; \Z)$, hence it admits a cup square.  (Here
we use the fact that $G$ is negative definite, hence $\det M_G\neq 0$,
so the restriction of any cohomology class from $X_G$ to its boundary
$Y_G$ is torsion.) We expect $\grad (U^j\otimes [K,E])$ to be
a rational number rather than an integer.

\begin{lem}
The boundary map decreases the Maslov grading by one.
\end{lem}
\begin{proof}
  We proceed separately for the two types of components of the
  boundary map.  After obvious simplifications we get that
\[
\grad (U^j\otimes [K,E])- \grad (U^j \cdot U^{a_v[K,e]}\otimes
[K,E-v])=
\]
\[
2g([K,E]) +\vert E\vert + 2a_v[K,E]-2g([K,E-v])-\vert
E-v\vert , 
\]
which, according to the definition of $a_v[K,E]$, is equal to 1.
Similarly, 
\[
\grad (U^j\otimes  [K,E])- \grad (U^j \cdot U^{b_v[K,e]}\otimes
[K+2v^*,E-v])=1
\]
follows from the same simplifications and Equation~\eqref{eq:refor}.
\end{proof}
It is not hard to see that the $J$-map preserves the Maslov grading. Indeed,
\[
\grad ([K,E])-\grad (J[K,E])=\grad ([K,E])-\grad ([-K-\sum _{v\in
  E}2v^*, E])=
\]
\[
=2g([K,E])-2g([-K-\sum _{v\in E}2v^*, E])+\frac{1}{4}(K^2-(-K-\sum
_{v\in E}2v^*)^2).
\]
Using the idenity of \eqref{eq:useful} and the alternative definition of $f(K,E)$, it follows that the
above difference is equal to zero.

Recall that the cardinality $\vert E \vert $ for a generator $[K,E]$ of $\CFmComb
(G)$ gives the $\delta$-grading, which decomposes each
$\CFmComb (G, \s )$ as
\[
\CFmComb (G, \s )=\oplus _{k=0}^{n} \CFmComb _k (G, \s ),
\]
where $n=\vert \Vertices (G)\vert$. It is easy to see that the differential 
$\partial $ decreases $\delta$-grading by one.

\subsection{Definition of the lattice homology}
We define the lattice homology groups as
follows.  Consider $({{\CFinfComb}}(G) ,\partial )$, and let
$({{\CFmComb}} (G), \partial ) $ denote the subcomplex generated by
those generators $U^j\otimes [K,E]$ for which $j\geq 0$ (and equipped with
the differential restricted to the subspace). Setting $U=0$ in this
subcomplex we get the complex $({{\CFaComb}} (G) , {\widehat {\partial }})$.
Obviously all these chain complexes split according to spin$^c$ structures
and admit a Maslov grading, $\delta$-grading and a $J$-map.
\begin{defn}
  Let us define the \emph{lattice homology} $\HFinfComb (G)$ as the homology
  of the chain complex $({{\CFinfComb}} (G), \partial )$. The homology of
  the subcomplex ${{\CFmComb}} (G)$ (with the boundary map $\partial$
  restricted to it) will be denoted by $\HFmComb (G)$, while the
  homology of $({{\CFaComb}} (G), {\widehat {\partial }})$ is $\HFaComb (G)$.
\end{defn}

Since the chain complex $\CFmComb (G)$ (and similarly, $\CFinfComb (G)$
and $\CFaComb (G)$) splits according to spin$^c$ structures, so does
the homology, giving the decomposition
\[
\HFmComb (G)=\oplus _{\s \in {\rm {Spin}}^c(Y_G)}\HFmComb (G, \s ).
\]
The $\delta$-grading then decomposes $\HFmComb (G, \s )$ further as 
\[
\HFmComb (G, \s )=\oplus _{k=0}^{n} \HFmComb _k (G, \s ),
\]
where $n=\vert \Vertices (G)\vert$. The Maslov grading provides an
additional ${\mathbb {Q}}$-grading on $\HFmComb (G, \s )$, but we reserve
the subscript
$\HFmComb_k(G,\s)$ for the $\delta$-grading.

\begin{rem}
  The embedding $i \colon \CFmComb (G)\to \CFinfComb (G)$ can be used
  to define a quotient complex $\CFpComb (G)$ (with the differential
  inherited from this construction) which fits into the short exact
  sequence
\[
0\to \CFmComb (G)\to \CFinfComb(G)\to \CFpComb (G) \to 0 .
\]
The homology of this quotient complex will be denoted by $\HFpComb
(G)$. The same splittings as before (according to spin$^c$ structures,
the $\delta$-grading and Maslov grading) apply to this theory is well.
The short exact sequence above then induces a long exact sequence on
the various homologies.

In a similar manner, $\CFmComb (G)$ and $\CFaComb (G)$ can be also 
connected by a short exact sequence:
\[
0\to \CFmComb (G)\stackrel{U}{\to} \CFmComb(G)\to \CFaComb (G) \to 0 ,
\] 
where the first map is multiplication by $U$. This short exact
sequence then induces a long exact sequence on homologies connecting
$\HFmComb (G)$ and $\HFaComb (G)$:
\[
\ldots \to \HFmComb _q (G)\stackrel{U}{\to }\HFmComb _{q} (G)\to
\HFaComb _{q} (G) \to \HFmComb _{q-1} \to \ldots 
\]
\end{rem}

\subsection{The structure of $\HFmComb (G)$}
The homology group $\HFmComb (G)$ is obviously an $\Field [U]$-module.
In the next result we describe an algebraic property these
particular modules satisfy.
\begin{thm} (N\'emethi, \cite{lattice})
\label{thm:structure}
  Suppose that $G$ is a negative definite plumbing tree and $\s$
  is a spin$^c$ structure on $Y_G$. Then the homology $\HFmComb
  (G, \s)$ is a finitely generated $\Field [U]$-module of the form
\[
\HFmComb (G, \s ) = \Field [U] \oplus \bigoplus _i A_i,
\]
where the modules $A_i$ are cyclic modules of the form $\Field [U]/(U^n)$.
Furthermore the $\Field [U]$-factor is in $\HFmComb _0(G,\s )$.
\end{thm}
The proof of the fact that $\HFmComb (G, \s)$ is finitely generated
(as an $\Field [U]$-module) is deferred until
Section~\ref{app:contract}.  Here we show how the previous discussion
and this finite generation implies the  rest of the
structure theorem.

Since any finitely generated $\Field [U]$-module is the direct sum of
cyclic $\Field [U]$-modules, in verifying Theorem~\ref{thm:structure}
we need to show that
\begin{itemize}
\item the $U$-torsion parts of $\HFmComb (G, \s )$ are all of the form
$\Field [U]/(U^n)$ and 
\item there is a single non-torsion module $\Field [U]$ in $\HFmComb
  (G, \s )$, and it lives in $\delta$-degree 0.
\end{itemize}
The first claim follows easily from the existence of a Maslov grading
and the fact that multiplication by $U$ drops this grading by $-2$:
these facts imply that the ideal $I$ in $A_i=\Field [U]/I$ should be
generated by a homogeneous polynomial, implying that $I=(U^n)$ for
some $n$.

For the second claim we define a further chain complex $(
\CFoverComb (G), {\overline {\partial }})$ associated to $G$ by
setting $U=1$ in the chain complex ${{\CFmComb }}(G)$ (or in
$\CFinfComb (G, \s)$). Then $\CFoverComb (G)$ is generated by the
pairs $[K,E]\in \Char (G)\times {\mathbb {P}}(\Vertices (G))$ over
$\Field$ (just like $\CFaComb (G)$ is), but the boundary map
${\overline {\partial}}$ is radically different from ${\widehat
  {\partial }}$. While in ${\widehat {\partial }}$ we allow nontrivial
boundary if and only if $a_v$ (or $b_v$) is equal to 0, in ${\overline
  {\partial}}$ the information captured by $a_v$ and $b_v$ is
completely lost. Therefore it is not surprising that

\begin{lem}
  \label{lem:UequalsOne}
  For a fixed spin$^c$ structure $\s$ the homology $\HFoverComb (G,
  \s)$ of $\CFoverComb (G, \s )$ is isomorphic to $\Field$, and
  $\HFoverComb (G, \s )=\HFoverComb _0(G, \s )$.
\end{lem}
\begin{proof}
By considering the set of pairs $[K,E]$ with 
spin$^c$ structure $\s_K=\s$, the corresponding hypercubes
\[
\{ K+\sum _{v_i\in E} 2t_iv_i^* \vert t_i\in [0,1]\}
\]
(when viewed $H^*(X_G; \Z )$ as subset of $H^* (X_G; \R )\cong \R ^n$)
provide a $CW$-decomposition of ${\mathbb {R}}^n$.  It follows from
the definition that $H_* (\CFoverComb (G, \s ), {\overline {\partial
  }})$ simply computes the $CW$-homology of ${\mathbb {R}}^n$, which
is equal (with $\Field$-coefficients) to $\Field$ in degree 0.
(Despite its simplicity, the $U=1$ theory turns out to be useful in
particular explicit computations.)
\end{proof}

\begin{proof}[Proof of Theorem~\ref{thm:structure}, assuming the finiteness
claim] Suppose that $\HFmComb (G, \s )$ is a finitely generated
  $\Field [U]$-module. We will appeal to the Universal Coefficient
  Theorem: notice that $\Field$ is an $\Field [U]$-module by defining
  the action of the polynomial $p(U)=\sum p_i U^i$ as multiplication
  by $\sum p_i$. Then
\[
0\to \HFmComb _q(G, \s ) \otimes _{\Field [U]}\Field \to \HFoverComb
_q(G, \s ) \to Tor (\HFmComb _{q-1}(G, \s ), \Field )\to 0
\]
proves the claim by the previos computation of $\HFoverComb (G, \s)$
and the facts that $Tor (\Field [U], \Field )=Tor (\Field [U]/U^n,
\Field )= 0$ and that $\Field [U]/(U^n)\otimes _{\Field [U]}\Field
=0$, while $\Field [U]\otimes _{\Field [U]}\Field =\Field$.
\end{proof}

\begin{cor}
  The $\Field [\uuinv ]$-module $\HFinfComb (G,\s)=\HFinfComb _0(G,\s
  )$ is isomorphic to $\Field [\uuinv ]$.
\end{cor}
\begin{proof}
By the Universal Coefficient Theorem we get that there is a short  
exact sequence
\[
0\to \HFmComb _q (G,\s )\otimes _{\Field [U]}\Field[\uuinv ]\to
\HFinfComb _q(G, \s) \to Tor (\HFmComb _{q-1}(G, \s ) , \Field [\uuinv ])\to 0.
\]
Since $Tor( \Field [U] , \Field [\uuinv ])=Tor (\Field [U]/(U^n),
\Field [\uuinv ])=0$ and $(\Field [U]/(U^n)) \otimes _{\Field
  [U]}\Field [\uuinv ]=0$, while $\Field [U] \otimes _{\Field
  [U]}\Field [\uuinv ]=\Field [\uuinv ]$, the claim obviously
follows. By Lemma~\ref{lem:UequalsOne} we get that the (single)
$\Field [U]$-factor is in $\HFmComb _0(G, \s)$, we get that
$\HFinfComb (G, \s ) = \HFinfComb _0(G, \s )$.
\end{proof}
\begin{defn}
  Let $\HFmComb _{red}(G, \s)\subseteq \HFmComb (G,\s) $ denote the kernel
  of the map $i _*$ induced by the embedding $i \colon \CFmComb (G,
  \s)\to \CFinfComb (G, \s)$. This group is finite dimensional as a
  vector space over $\Field$ and is called the \emph{reduced lattice
    homology} of $(G, \s )$.
\end{defn}

\subsection{Examples}
We conclude this section by working out a simple example which will be
useful in our later discussions.
\begin{exa}\label{ex:minusz1}
Suppose that the tree $G$ has a single vertex $v$ with framing $-1$.
The chain complex $\CFinfComb (G)$
is generated over $\Field [\uuinv ]$ by the elements
\[
\{ [2n+1, \{ v\}], [2n+1, \emptyset ]\mid n\in \Z \},
\]
where a characteristic vector on $G$ is denoted by its value $2n+1$ on
$v$. The boundary map on $[2n+1, \emptyset ]=[2n+1]$ 
is given by $\partial [2n+1]=0$ and by
  \[ \partial [ 2n+1, \{ v\}]
  = \left\{\begin{array}{ll}
      [2n+1]  + U^n \otimes [2n-1] & {\text{if $n\geq 0$}} \\
      U^{-n}\otimes [2n+1]+[2n-1] & {\text{if $n< 0$}}. 
      \end{array}
    \right.\] These formulae also describe the chain complexes
    $\CFmComb (G)$ and $\CFaComb (G)$ (generated over $\Field [U]$ and
    over $\Field$). Let us consider the map $F$ from $\CFinfComb (G)$
    to the subcomplex $\Field [\uuinv ]\langle [-1]\rangle \subseteq
    \CFinfComb (G)$ generated by the element $[-1]$, defined as
 \[  F([ 2n+1, E])
  = \left\{\begin{array}{ll}
      0 & {\text{if $E=\{ v\}$}} \\
      U^{\frac{1}{2}n(n+1)}\otimes [-1] & {\text{if $E=\emptyset $}}
      \end{array}
    \right.\]
This map provides a chain homotopy equivalence between 
$\CFinfComb (G)$ and $\Field [\uuinv ]$ (the latter equipped with the
differential $\partial =0$), as shown by the chain homotopy
\[ 
H([ 2n+1, E]) = \left\{\begin{array}{ll}
    0 & {\text{if $E=\{ v\}$ or $n=-1$}} \\
    \sum _{i=0}^n U^{s_i}\otimes [2(n-i)+1, v] & {\text{if
        $E=\emptyset$ and $n\geq 0$}}\\
    \sum _{i=0}^{-n-2}U^{r_i}\otimes [2(n+i+1)+1, v] & {\text{if
        $E=\emptyset $ and $n< -1$}}\\
      \end{array}
    \right.\] where $s_0=0$ and $s_{i}=s_{i-1}+b_{v}[2(n-i-1)-1,
    v]=\frac{1}{2}i(2n+1-i)$, $r_0=0$ and $r_{i}=r_{i-1}+a_v[2(n+i)+1,
    v]=-\frac{1}{2}i(2n+1+i)$.  In conclusion, the homology
    $\HFinfComb (G)$ (and similarly $\HFmComb (G)$ and $\HFaComb (G)$)
    is generated by the class of $[-1]$ over $\Field [\uuinv ]$ (and
    over $\Field [U]$ and $\Field$, respectively). In particular,
    $\HFmComb _i (G)=0$ for $i>0$.
\end{exa}
\begin{rem}\label{rem:gk}
  A similar computation shows that the lattice homology $\HFmComb
  (G_k)$ of the graph $G_k$ we get by considering a linear chain of
  $k$ vertices of framing $(-2)$ and a final one with framing $(-1)$
  (cf. Figure~\ref{fig:lanc}) is also isomorphic to $\Field [U]$ (and
  to $\Field $ in the $\CFaComb$-theory). The above example discusses
  the case $k=0$ of this family. We will provide details of the
  computation for further $k$'s in Section~\ref{sec:app}.
\begin{figure}[ht]
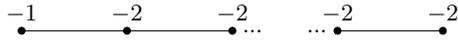

\begin{center}
\setlength{\unitlength}{1mm} \unitlength=0.7cm
\begin{graph}(2,0.5)(-2,0)
\graphnodesize{0.15}

\roundnode{n1}(-3,0)
\roundnode{n2}(-1,0)
\roundnode{n3}(1,0)
\roundnode{n4}(3,0)
\roundnode{n5}(5,0)

\edge{n1}{n2}
\edge{n2}{n3}
\edge{n4}{n5}

\autonodetext{n1}[n]{\small $-1$} 
\autonodetext{n2}[n]{\small $-2$} 
\autonodetext{n3}[n]{\small $-2$}
\autonodetext{n4}[n]{\small $-2$}
\autonodetext{n5}[n]{\small $-2$}
\autonodetext{n3}[e]{\small $...$}
\autonodetext{n4}[w]{\small $...$}
\end{graph}
\end{center}
\caption{{\bf The plumbing tree $G_k$.} The graph has $k+1$ vertices,
  the left-most admitting framing $(-1)$ while all the others have
  framing $(-2)$. It is easy to see that the corresponding 3-manifold
  is $S^3$.}
\label{fig:lanc}
\end{figure}
\end{rem}

Recall that for the disjoint union $G=G_1\cup G_2$ of two trees/forests
the chain complex of $G$ (and therefore the lattice homology of $G$)
splits as the tensor product of the lattice homologies of $G_1$ and $G_2$
(over the coefficient ring of the chosen theory). As a quick corollary we
get 
\begin{cor}\label{c:kulonegy}
Suppose that $G=G_1\cup G_2$ where $G_2$ is the graph encountered in
Example~\ref{ex:minusz1}. Then $\HFmComb (G)\cong \HFmComb (G_1)$.
(Similar statements hold for the other versions of the theory.)
\end{cor}
\begin{proof}
  By the connected sum formula (Equation~\eqref{eq:ConnectedSum}), and
  by the computation in Example~\ref{ex:minusz1} we get that
\[
\HFmComb (G) \cong \HFmComb (G_1)\otimes _{\Field [U]}\HFmComb (G_2)\cong
\HFmComb (G_1)\otimes _{\Field [U]}{\Field [U]}\cong \HFmComb (G_1).
\]
verifying the statement.
\end{proof} 

\section{The knot filtration on lattice homology}
\label{sec:knots}
Denote the vertices of the tree $\Gamma _{v_0}$ by $V= \Vertices
(\Gamma _{v_0})=\{ v_0, v_1, \ldots , v_n\}$.  Assume that each $v_j$
with $j>0$ is equipped with a framing $m_j\in \Z$, but leave the
vertex $v_0$ unframed.  In the following we will assume that $G=\Gamma
_{v_0}-v_0$ is negative definite. 
The reason for this assumption is that for more general graphs
lattice homology provides groups isomorphic to the corresponding
Heegaard Floer homology groups only after completion; in particular
after allowing infinite sums in the chain complex. For such elements, 
however, the definition of any filtration requires more care. To avoid these
technical difficulties, here we restrict ourselves to the negative definite
case.

For a framing $m_0\in \Z$ on $v_0$ denote the framed graph we get from
$\Gamma _{v_0}$ by $G_{v_0}=G_{m_0}(v_0)$.  (We will always assume
that $m_0$ is chosen in such a way that $G_{m_0}(v_0)$ is also
negative definite.)  Let $\Sigma \in H_2(X_{G_{v_0}}; {\mathbb {Q}})$
be a homology class satisfying:
\begin{equation}
  \label{eq:DefSigma}
\Sigma=v_0+ \sum_{j=1} ^n a_j \cdot v_j \quad (\text{where~}a_j\in {\mathbb {Q}}),
\quad {\mbox {and}} \quad v_j \cdot \Sigma = 0 \quad (\text{for all~}j>0).
\end{equation}
Notice that since $G=\Gamma _{v_0}-v_0$ is assumed to be negative
definite, the class $\Sigma$ exists and is unique. In the next two
section we will follow the convention that characteristic classes on
$G$ and subsets of $V-\{ v_0\}$ will be denoted by $K$ and $E$
respectively, while the characteristic classes on $G_{v_0}$ and
subsets of $V$ will be denoted by $L$ and $H$, resp.

\begin{lem}\label{lem:letezik}
  Let us fix a generator $[K,E]\in \Char (G)\times {\mathbb {P}}
  (V-v_0)$ of the lattice  chain complex $\CFinfComb (G)$ of
  $G$. There is a unique element $L=L_{[K,E]}\in \Char
  (G_{v_0})$ with the properties that for $H_E=E\cup \{ v_0\}$
\begin{itemize}
\item $L\vert _G =K$ and
\item $a_{v_0}[L, H_E]=
b_{v_0}[L, H_E]=0$.
\end{itemize}
\end{lem}
\begin{proof}
The equality $a_{v_0}[L, H_E]=b_{v_0}[L,H_E]$ is,
by definition, equivalent to $A_{v_0}([L, H_E])=
B_{v_0}([L, H_E ])$. By its definition 
$A_{v_0}([L, H_E])=g([K,E])$ is independent of 
$L(v_0)$ (and of the framing $m_0=v_0^2$ of $v_0$), while since
$K(v_j)=L(v_j)$ for $j>0$, by Equation~\ref{eq:refor}
\[
2B_{v_0}([L, H_E ])= L(v_0)+v_0^2+2g([K+2v_0^*, E]).
\] 
The identity $2A_{v_0}([L, H_E])=2B_{v_0}([L , H_E ])$ then uniquely
specifies $L(v_0)$:
\begin{align*}
L(v_0)&=-v_0^2+ 2g([K,E])-2g([K+2v_0^*, E]) \\
&=-v_0^2+\min_{I\subset E} \left(\sum _{v\in I } K(v)+(\sum _{v\in
  I}v)^2\right)- 
\min_{I\subset E}\left(\sum _{v\in I } K(v)+(\sum _{v\in I}v)^2+2v_0\cdot
(\sum _{v\in I}v)\right).
\end{align*}
Since $K$ is characteristic, both minima are even, and therefore
$L(v_0)\equiv v_0^2 \mod 2$, implying that $L$ is also characteristic.
\end{proof}

\begin{defn}
  We define the \emph{Alexander grading} $A([K,E])$ of a generator
  $[K,E]$ of $\CFinfComb (G)$ by the formula
\[
A([K,E])=\frac{1}{2}(L(\Sigma )+\Sigma ^2)\in {\mathbb {Q}},
\]
where $L=L_{[K,E]}$ is the extension of $K$ found in
Lemma~\ref{lem:letezik} and $\Sigma $ is the (rational) homology
element in $H_*(X_{G_{v_0}}; {\mathbb {Q}})$ associated to $v_0$ 
in Equation~\eqref{eq:DefSigma}.
(In the above formula we regard $L\in H^2 (X_{G_{v_0}}; \Z )$ as a
cohomology class with rational coefficients.) Notice that since $v_j
\cdot \Sigma =0$ for all $j>0$, the above expression is equal to
$\frac{1}{2}( L(\Sigma )+v_0\cdot \Sigma )$.

We extend this grading to expressions of the form $U^j\otimes [K,E]$
with $j\in \Z$ by
$$A(U^j\otimes [K,E])=-j+A([K,E]).$$
\end{defn}
In the definition above we fixed a framing $m_0$ on $v_0$, and it is
easy to see that both the values of $L(v_0)$ and of $\Sigma ^2=
v_0\cdot \Sigma $ depend on this choice.

\begin{lem}
  The value $A([K,E])$ is independent of the choice of the framing
  $m_0=v_0^2$ of $v_0$.
\end{lem}
\begin{proof}
  By the identities of Lemma~\ref{lem:letezik} it is readily visible
  that $L(v_0)$ (and hence $L(\Sigma )$) changes by $-1$ if $v_0^2$ is
  replaced by $v_0^2+1$. Since $\Sigma ^2 $ changes exactly as $v_0^2$
  does, the sum $L(v_0)+\Sigma ^2$ (and hence $\frac{1}{2}(L(\Sigma )
  +\Sigma ^2)$) does not depend on the chosen framing $v_0^2$ on $v_0$.
\end{proof}

Since $\Sigma$ is not an integral homology class, there is no reason
to expect that $A([K,E])$ is an integer in general.  On the other
hand, it is easy to see that if $K,K'$ represent the same spin$^c$
structure then $A([K,E])-A([K',E'])$ is an integer:
if $K'=K+2y^*$ (with $y\in H_2 (X_G; \Z )$) then 
\[
A([K,E])-A([K', E'])=\frac{1}{2}(L_{[K,E]}-L_{[K',E']})(v_0)\in \Z
\]
since $y\cdot \Sigma =0$ and both $L_{[K,E]}$ and $L_{[K',E']}$ are
characteristic cohomology classes.
\begin{defn}\label{def:mod1}
  For each spin$^c$ structure $\s$ of $G$ there is a rational number
  $i_{\s}\in [0,1)$ with the property that mod 1 the Alexander grading
  $A([K,E])$ for a pair $[K,E]$ with $\s _K=\s$ is congruent to
  $i_{\s}$.
\end{defn}

\begin{defn}
  The Alexander grading $A$ of generators  naturally
  defines a filtration $\{ {\mathcal {F}}_i\} $ on the chain complex
  $\CFinfComb (G)$ (which we will still denote by $A$ and will call
  the \emph{Alexander filtration}) as follows: an element $x\in
  \CFinfComb (G)$ is in ${\mathcal {F}}_i$ if every component of $x$
  (when written in the $\Field$-basis $U^j\otimes [K,E]$) has
  Alexander grading at most $i$. Intersecting the above filtration
  with the subcomplex $\CFmComb (G)$ we get the Alexander filtration
  $A$ on $\CFmComb (G)$. Similarly, the definition provides Alexander
  filtrations on the chain complexes $\CFaComb (G)$ and $\CFpComb
  (G)$.
\end{defn}
Equipped with the Alexander filtration, now $(\CFinfComb (G), \partial
)$ is a filtered chain complex, as the next lemma shows.
\begin{lem}\label{lem:filtralt}
  The chain complex $\CFinfComb (G)$ (and similarly, $\CFmComb (G)$
  and $\CFaComb (G)$) equipped with the Alexander filtration $A$ is a
  filtered chain complex, that is, if $x\in {\mathcal {F}}_i$ then
  $\partial x\in {\mathcal {F}}_i$. 
  \end{lem}
\begin{proof} We need to show that for a generator $[K,E]$ the
  inequality $A(\partial [K,E])\leq A([K,E])$ holds.  Recall that
  $\partial [K,E]$ is the sum of two types of elements. In the
  following we will deal with these two types separately, and verify a
  slightly stronger statement for these components.
  
  Let us first consider the component of the boundary of the shape of
  $U^{a_v[K,E]}\otimes [K,E-v]$ for some $v\in E$.  We claim that in
  this case
  \begin{equation} \label{eq:alexkul}
  A([K,E])-A(U^{a_v[K,E]}\otimes [K,E-v])=a_v[K+2v_0^*, E]
  \end{equation}
  holds, obviously implying that the Alexander grading of this
  boundary component is not greater than that of $[K,E]$.  To verify
  the identity of \eqref{eq:alexkul}, write $\Sigma $ as $v_0+\sum _{j=1}^n
  a_j\cdot v_j$, and note that twice the left-hand-side of
  Equation~\eqref{eq:alexkul} is equal to
\[
K(\sum_{j=1}^n a_j\cdot v_j)+L_{[K,E]}(v_0)+\Sigma
^2+2g([K,E-v])-2g([K,E])-K(\sum_{j=1}^n a_j\cdot v_j)-
L_{[K,E-v]}(v_0)-\Sigma ^2,
\]
which, after the simple cancellations and the extensions found in
Lemma~\ref{lem:letezik} is equal to
 \[
 2g([K,E])-2g([K+2v_0^*,
 E])+2g([K,E-v])-2g([K,E])-2g([K,E-v])+2g([K+2v_0^*, E-v].
 \]
 After further cancellations, this expression gives $2a_v[K+2v_0^*, E]$, verifying
 Equation~\eqref{eq:alexkul}. Since
 $a_v\geq 0$, Equation~\eqref{eq:alexkul} concludes the
 argument in this case.

Next we compare the Alexander grading of the term $U^{b_v[K,E]}\otimes
[K+2v^*,E-v]$ to $A([K,E])$. Now we claim that
\begin{equation}\label{eq:alexkulmasodik}
A([K,E])-A(U^{b_v[K,E]}\otimes [K+2v^*,E-v])=b_v[K+2v_0^*, E].
\end{equation}
As before, after substituting the defining formulae into the terms of
twice the left-hand-side of \eqref{eq:alexkulmasodik} we get
\[
K(\sum _{j=1}^n a_j\cdot v_j)+L_{[K,E]}(v_0)+\Sigma ^2 +2B_v[K,E]-2g([K,E]) -
\]
\[
-(K+2v^*)(\sum _{j=1}^n a_j\cdot v_j) -L_{[K+2v^*, E-v]}(v_0)-\Sigma ^2 .
\]
From the fact that $v^*(\Sigma )=0$ we get that $2v^*(\sum _{j=1}^n a_j
\cdot v_j)=-2v\cdot v_0$, hence by considering the form of $B_v$ given
in \eqref{eq:refor} we get that this term is equal to
\[
2g([K,E])-2g([K+2v_0^*, E])+2g([K+2v^*, E-v])+K(v)+v^2+2v\cdot v_0-2g([K,E])-
\]
\[
-2g([K+2v^*, E-v])
+2g([K+2v^*+2v_0^*, E-v]),
\]
and this expression is obviously equal to $2b_v[K+2v_0^*, E]$. Once again,
since $b_v\geq 0$, the statement of the lemma follows.
\end{proof}

\begin{defn}
  We define the filtered chain complex $(\CFinfComb (G), \partial ,
  A)$ (and similarly $(\CFmComb (G), \partial , A)$ and $(\CFaComb
  (G), \partial , A)$) the \emph{filtered lattice chain complex} of
  the vertex $v_0$ in the graph $\Gamma _{v_0}$.
\end{defn}

\begin{rem}
Recall that the chain complex $\CFmComb (G)$ splits according to the spin$^c$ 
structures of the 3-manifold $Y_G$. By intersecting the Alexander filtration
with the subcomplexes $\CFmComb (G, \s )$ for every spin$^c$ structure
$\s$, we get a splitting of the filtered chain complex according 
to spin$^c$ structures as well. The same remark applies to the
$\CFinfComb$ and $\CFaComb$ theories.
\end{rem}

\begin{defn}
  The \emph{knot lattice homology} $\HFKmComb (\Gamma _{v_0})$ (and
  $\HFKinfComb (\Gamma _{v_0}), \HFKaComb (\Gamma _{v_0})$) of $v_0$
  in the graph $\Gamma _{v_0}$ is defined as the homology of the
  graded object associated to the filtered chain complex $(\CFmComb
  (G), \partial , A)$ (and of $(\CFinfComb (G), \partial , A)$,
  $(\CFaComb (G), {\widehat {\partial }}, A)$ respectively).  As
  before, the groups $\HFKmComb (\Gamma _{v_0})$ (and similarly
  $\HFKinfComb (\Gamma _{v_0})$ and $\HFKaComb (\Gamma _{v_0})$) split
  according to the spin$^c$ structures of $Y_G$, giving rise to the groups
  $\HFKmComb (\Gamma _{v_0}, \s )$ for $\s \in$Spin$^c(Y_G)$.
\end{defn}

Let us fix a spin$^c$ structure $\s$ on $Y_G$. The group 
$\HFKmComb (\Gamma _{v_0}, \s )$ then splits according to the Alexander 
gradings as
\[
\oplus _a \HFKmComb (\Gamma _{v_0}, \s , a),
\]
and the components $\HFKmComb (\Gamma _{v_0}, \s , a)$ are further
graded by the absolute $\delta$-grading (originated from the cardinality
of the set $E$ for a generator $[K,E]$) and by the Maslov grading.

The relation between the Alexander filtration and the $J$-map is given by
the following formula:
\begin{lem}
$A(J[K,E]))=-A([K-2v_0^*, E])$.
\end{lem}
\begin{proof}
  Recall that $J[K,E]=[-K-\sum _{v\in E}2v^*, E]$.  With the extension
  $L$ of $-K-\sum _{v\in E}2v^*$ given by Lemma~\ref{lem:letezik}
  (with the convention that $v_0^2=0$) we have that
\[
2A(J[K,E])=(-K-\sum _{v\in E}2v^*)(\Sigma - v_0)+L(v_0)+\Sigma ^2.
\]
Since $v^*(\Sigma )=0$, by the definition of $L(v_0)$ and the identity of Remark~\ref{rem:altfor}
this expression is equal to 
\[
-K(\Sigma -v _0)+2v_0\cdot (\sum _{v\in E}v)+\Sigma ^2+2g[K,E]-2f[K,E]-2g[K-2v_0^*, E]+
2f[K-2v_0^*, E].
\]
With the same argument the identity
\[
2A([K-2v_0^*, E])=K(\Sigma -v_0)-2v_0^*(\Sigma -v_0)+L'(v_0)+\Sigma
^2=
\]
\[
=K(\Sigma -v_0)-\Sigma ^2+2g[K-2v_0^*, E]-2g[K,E]
\]
follows (since $v_0\cdot \Sigma =\Sigma ^2$ and $v_0^2=0$). Now the identity of the lemma
follows from the observation that $f[K,E]-f[K-2v_0^*,E]-v_0\cdot (\sum _{v\in E}v)=0$.
\end{proof}

Define $J_{v_0}\colon \CFinfComb (G)\to \CFinfComb (G)$ by the 
formula
\[
[K,E]\mapsto [-K-\sum _{u\in E}2u^*-2v_0^*, E],
\]
on a generator $[K,E]$ and extend $U$-equivariantly and linearly to
$\CFinfComb (G)$. It is easy to see that $J_{v_0}^2=Id$. The result of
the previous lemma can be restated as
\[
A(J_{v_0}[K,E])=-A[K,E].
\]
This map is similar to the $J$-map, but takes the vertex $v_0$ into
special account. For the next statement recall from
Definition~\ref{def:mod1} the quantity $i _{\s }$ associated to a
spin$^c$ structure $\s$ on $G$.
\begin{lem} \label{lem:ujmap}
The map sending the generator $[K,E]\in \CFinfComb (G, \s )$
to $U^{i_{\s}-A([K,E])}J_{v_0}[K,E]$ is a chain map.
\end{lem}
\begin{proof}
  We show first that the application of the above map to
  $U^{a_v[K,E]}\otimes [K,E-v]$ for some $v\in E$ is equal to
\[
U^{i_{\s}-A([K,E])}\cdot U^{b_v[-K-\sum _{u\in E}2u^*-2v_0^*,
  E]}\otimes [-K-\sum _{u\in E}2u^*-2v_0^*+2v^*, E-v] .
\]
The identification of $J_{v_0}(U^{a_v[K,E]}\otimes [K,E])$ with the
above term easily follows from the observation that
\begin{equation}\label{eq:kompi}
a_v[K,E]+i_{\s }-A([K,E-v])=i_{\s}-A([K,E])+b_v[-K-\sum _{u\in E}2u^*-2v_0^*].
\end{equation}
Equation~\eqref{eq:kompi}, however, is a direct consequence of the
equality $b_v[-K-\sum _{u\in E}2u^*-2v_0^*]=a_v[K+2v_0^*, E]$ and the
definitions of the terms describing the Alexander gradings. A similar
computation shows the identity for the other type of boundary
components (involving the terms of the shape
$U^{b_v[K,E]}\otimes [K+2v^*, E-v]$), concluding the proof.
\end{proof}

\begin{exas}\label{ex:trivik}
Two examples of the filtered chain complexes associated to certain  
graphs  can be determined as follows.
\begin{itemize}
\item Consider first the graph $\Gamma _{v_0}$ with two vertices
$\{v_0, v\}$, connected by a single edge, and with $(-1)$ as the framing
of $v$. The chain complex of $G=\Gamma _{v_0}-v_0$ has been determined
in Example~\ref{ex:minusz1}. 
A straightforward calculation  shows that $A([2n+1])=n+1$ and 
  \[ A( [ 2n+1, \{ v\}])
  = \left\{\begin{array}{ll}
       n+1 & {\text{if $n\geq 0$}} \\
       n & {\text{if $n< 0$}}.
      \end{array}
    \right.\] This formula then describes the Alexander filtration on
    $\CFmComb (G)$.  (Recall that $A(U^i \otimes [K,E])=-j
    +A([K,E])$.)  It is easy to see that the chain homotopy
    encountered in Example~\ref{ex:minusz1} respects this Alexander
    filtation, hence the filtered lattice chain complex $(\CFinfComb
    (G), A)$ is filtered chain homotopic to $\Field [\uuinv ]$,
    generated by the element $g$ in filtration level 0.  In
    conclusion, $\HFKaComb (\Gamma _{v_0})$ and $\HFKmComb (\Gamma
    _{v_0})$ are both generated by the element $[-1]$ (over $\Field $
    and $\Field [U]$, respectively), and the Alexander and Maslov
    gradings of the generator are both equal to $0$.

  \item In the second example consider the graph $\Gamma '_{v_0}$ on the same
    two vertices $\{ v_0, v\}$, now with no edges at all.  (That is,
    $\Gamma '_{v_0}$ is given from $\Gamma _{v_0}$ by erasing the
    single edge of $\Gamma _{v_0}$.) The background graph $G$ (and
    hence the chain complex $\CFmComb (G)$) is obviously the same as
    in the first example, but the Alexander grading $A'$ is much simpler
    now: $A'([2n+1])=A'([2n+1, \{ v\}])=0$ for all $n\in \Z$.  Once
    again, the chain homotopy of Example~\ref{ex:minusz1} is a
    filtered chain homotopy, hence we can apply it to determine the
    filtered lattice chain complex of $\Gamma _{v_0}'$, concluding
    that $(\CFinfComb (G), A')$ is filtered chain homotopic to $\Field
    [\uuinv ]$ with the generator in Alexander grading $0$.  Once
    again $\HFKmComb (\Gamma _{v_0}')$ is generated by $[-1]$.
  \end{itemize}
  In conclusion, the filtered chain complexes of the two examples are
  filtered chain homotopic to each other.
\end{exas}
\begin{rem} \label{rem:gkknot}
  Let $\Gamma ^k_{v_0}$ be constructed from $G_k$ of Remark~\ref{rem:gk} by
  attaching to it the vertex $v_0$ together with the edge connecting
  $v_0$ and the single $(-1)$-framed vertex.  Minor modifications of
  the argument above identifies the filtered lattice chain complex of
  $\Gamma ^k _{v_0}$ with $\Field [\uuinv ]$ (with the generator
  having Alexander grading 0). We will return to this example in 
  Section~\ref{sec:app}.
\end{rem}

\section{The master complex and the connected sum formula}
\label{sec:master}
As we will see in the next section, the filtered chain complexes
defined in the previous section (together with certain maps, to be
discussed below) contain all the relevant information we need for
calculating the lattice homologies of graphs we get by attaching
various framings to $v_0$.  The Alexander filtration $A$ on
$\CFinfComb (G)$ can be enhanced to a double filtration by considering
the double grading
\begin{equation}\label{eq:duplafil}
U^j \otimes [K,E]\mapsto (-j, A(U^j\otimes [K,E])) .
\end{equation}
In fact, this doubly filtered chain complex determines (and is
determined by) the filtered chain complex $(\CFmComb (G), A)$.  
Notice that multiplication by $U$ decreases Maslov grading by 2, $-j$
by 1 and Alexander grading by 1. 

In describing the further structures we need, it is slightly more
convenient to work with $\CFinfComb (G)$, and therefore we will
consider the doubly filtered chain complex above. In the following we
will find it convenient to equip $\CFinfComb (G)$ with the following
map. 
\begin{defn}\label{def:mapn}
The map $N \colon \CFinfComb (G) \to \CFinfComb (G)$ is defined by the
formula
\begin{equation}\label{eq:nmap}
N(U^j\otimes [K,E])= U^{i_{s_K}-A[K,E]+j}\otimes [K+2v_0^*, E] .
\end{equation}
\end{defn}
Notice that $N$ does not preserve the spin$^c$ structure of a given
element. Indeed, if $\s _{v_0}$ denotes the spin$^c$ structure we get
by twisting $\s$ with $v_0^*$ (and hence we get $c_1(\s _{v_0^*})=c_1
(\s )+2v_0^*$) then $N$ maps $\CFinfComb (G, \s )$ to $\CFinfComb (G,
\s _{v_0})$. In fact, by choosing another rational number $r$ (with
$r\equiv i_{\s _K} \mod 1$) instead of $i_{\s _K}$ in the above
formula, we get only multiples of $N$ (multiplied by appropiate monoms of
$U$).
\begin{lem}
  The map $N$ is a chain map, and provides an isomorphism between the
  chain complex $\CFinfComb (G, \s)$ and $\CFinfComb (G, \s
  _{v_0^*})$.
\end{lem}
\begin{proof} The fact that $N$ is a chain map follows from the 
identities
\begin{equation}\label{eq:akul}
a_v[K,E]-A([K,E-v])=a_v[K+2v_0^*, E]-A([K,E])
\end{equation}
and 
\begin{equation}\label{eq:bkul}
b_v[K,E]-A([K+2v^*, E-v])=b_v[K+2v_0^*, E]-A([K,E]).
\end{equation}
These identities follow easily from the definitions of the terms. To
show that $N$ is an isomorphism, let the spin$^c$ structure 
$\s _{-v_0^*}$ the denoted by $\t$ and consider the map
\[
M(U^j \otimes [K,E])=U^{A([K-2v_0^*, E])+j-i_{\t}}\otimes
[K-2v_0^*, E] .
\]
$M$ is also a chain map (as the identities similar to
\eqref{eq:akul} and \eqref{eq:bkul} show), and $M$ and $N$ are
inverse maps. It follows therefore that $N$ is an isomorphism between chain
complexes.
\end{proof}
Notice that $N$ can be written as the composition of the $J$-map with
the map $U^{i_{\s}-A([K,E])}J_{v_0}[K,E]$ considered in
Lemma~\ref{lem:ujmap}.

\begin{defn}
  Suppose that for $i=1,2$ the triples $(C_i, A_i, j_i)$ are doubly
  filtered chain complexes and $N_i\colon C_i \to C_i$ are given
  maps. Then the map $f\colon C_1\to C_2$ is an \emph{equivalence} of
  these structures if $f$ is a (doubly) filtered chain homotopy
  equivalence commuting with $N _i$, that is, $f\circ N_1=N_2\circ f$.
\end{defn}
With this definition at hand, now we can define the \emph{master complex} of 
$\Gamma _{v_0}$ as follows. 
\begin{defn}
  Suppose that $\Gamma _{v_0}$ is given. Consider $\CFinfComb (G)$
  with the double filtration $(-j,A)$ as above, together with the map
  $N$ defined in Definition~\ref{def:mapn}. The equivalence class of
  the resulting structure is the \emph{master complex} of $\Gamma
  _{v_0}$.
\end{defn}
As a simple example, a model for the master complex for each of the
two cases in Example~\ref{ex:trivik} can be easily determined:
regarding the map $U^j \otimes [K,E]\mapsto (-j , A (U^j\otimes
[K,E]))$ as a map into the plane, (a representative of) the master
complex will have a $\bfz _2$ term for each coordinate $(i,i)$, and
all other terms (and all differentials) are zero.  In addition, the
map $N$ in this model is equal to the identity.  (Note that in this
case the background 3-manifold is diffeomorphic to $S^3$, hence admits
a unique spin$^c$ structure.)  In short, the master complex for both
cases in Example~\ref{ex:trivik} is $\Field [\uuinv ]$, with the
Alexander grading of $U^j$ being equal to $j$ and with $N=id$.

Obviously, by fixing a spin$^c$ structure $\s \in Spin^c(Y_G)$ we can
consider the part $\MCFinfComb (\Gamma _{v_0}, \s )$ of the master
complex generated by those elements $U^j\otimes [K,E]$ which satisfy
the constraint $\s _K=\s$. As we noted earlier, $N$ maps components of
the master complex corresponding to various spin$^c$ structures into
each other.

\subsection{The connected sum formula}
Suppose that $\Gamma _{v_0}$ and $\Gamma '_{w_0}$ are two graphs with
distinguished vertices $v_0, w_0$. Their connected sum is defined in the following:
\begin{defn}
    Let  $\Gamma _{v_0}$ and $\Gamma '_{w_0}$ be two graphs with distinguished
  vertices $v_0$ and $w_0$. Their {\em connected sum} is the graph
  obtained by taking the disjoint union of $\Gamma_{v_0}$ and $\Gamma'_{w_0}$,
  and then identifying the distinguished vertices $v_0=w_0$. The
  resulting graph
  \[\Delta_{(v_0=w_0)}=\Gamma_{v_0}\#_{(v_0=w_0)} \Gamma'_{w_0}\]
  (which will be
  a tree/forest provided both $\Gamma _{v_0}$ and $\Gamma '_{w_0}$
  were trees/forests) has  a distinguished vertex $v_0=w_0$.
\end{defn}
\begin{rem}
  Notice that this construction gives the connected
  sum of the two knots specified by $v_0$ and $w_0$ in the two 3-manifolds
  $Y_G$ and $Y_{G'}$.
\end{rem}

Recall that for the disjoint graphs $G=\Gamma _{v_0}-v_0$ and
$G'=\Gamma ' _{w_0}-w_0$ the chain complex $\CFinfComb (G\cup G')$ of
their connected sum is simply the tensor product
of $\CFinfComb (G)$ and 
$\CFinfComb (G')$ (over $\Field [\uuinv ]$). We will
denote the Alexander grading/filtration on 
$\CFinfComb (G)$ by $A_{v_0}$ and on
$\CFinfComb (G')$ by $A_{w_0}$.
\begin{thm}\label{thm:alexosszead}
For the Alexander grading $A_{\#}$ of the generator
$[K_1, E_1]\otimes [K_2, E_2] \in \CFinfComb (G\cup G')$ 
induced by the distinguished vertex $v_0=w_0$ in $\Delta _{(v_0=w_0)}$
we have that
\[
A_{\#} ([K_1, E_1]\otimes [K_2, E_2])=A_{v_0}([K_1, E_1])+A_{w_0}([K_2, E_2]).
\]
\end{thm}
\begin{proof}
  For simplicity fix $v_0^2=w_0^2=0$ and consider $\Sigma _{v_0}$ and
  $\Sigma _{w_0}$ on the respective sides of the connected sum. By the
  calculation from Lemma~\ref{lem:letezik} it follows that for
  the extensions $L_i$ of $K_i$ over the distinguished points $v_0,
  w_0$, and extension $L$ over $v_0=w_0$ we have
\[
L _{E_1\cup E_2}(v_0=w_0)=
(L_1)_{E_1}(v_0)+(L_2)_{E_2}(w_0).
\]
Since $\Sigma ^2 _{v_0=w_0}=(\Sigma _{v_0}+\Sigma _{w_0})^2=\Sigma ^2_{v_0}+\Sigma
^2_{w_0}$, the above equality shows that both terms of the defining
equation of the Alexander grading are additive, concluding the result.
\end{proof}
As a corollary, we can now show that
\begin{thm}
The master complexes of $\Gamma _{v_0}$ and $\Gamma '_{w_0}$ determine
the master complex of the connected sum
$\Delta _{(v_0=w_0)}$.
\end{thm}
\begin{proof}
  As we saw above, the chain complexes for $\Gamma _{v_0}$ and $\Gamma
  ' _{w_0}$ determine the chain complex of $\Delta _{(v_0=w_0)}$ by
  taking their tensor product. This identity immediately shows that
  the $j$-filtration on the result is determined by the
  $j$-filtrations on the components. The content of
  Theorem~\ref{thm:alexosszead} is that the Alexander filtration on
  the connected sum is also determined by the Alexander filtrations of
  the pieces.  Finally, the map $N$ is built from the maps $J$ and
  $J_{v_0}$, which simply add for the connected sum, implying the
  result. A minor adjustment is needed in the last step: if $i _{\s }$
  and $i_{\s '}$ are the rational numbers determined by
  Definition~\ref{def:mod1} for the spin$^c$ structures $\s$ and $\s
  '$, then for $\s \# \s '$ we take either their sum (if it is in
  $[0,1)$) or $i_{\s }+i _{\s '}-1$.
\end{proof}
As a simple application of this formula, consider a
graph $\Gamma _{v_0}$ and associate to it two further graphs as
follows. Both graphs are obtained by adding a further element $e$ to $\Vertices (\Gamma _{v_0})$,
equipped with the framing $(-1)$. We can proceed in the following two ways:
\begin{enumerate}
\item Construct $\Gamma _{v_0}^+$ by adding an edge connecting $e$ and
  $v_0$ to $\Gamma _{v_0}$.
\item  Define $\Gamma _{v_0} ^d$ by simply adding  $e$ (with the fixed framing $(-1)$)
without adding any extra edge.
\end{enumerate}
For a pictorial presentation of the two graphs, see
Figure~\ref{fig:ketgraf}. It is easy to see that $\Gamma _{v_0}^+$ is
the connected sum of $\Gamma _{v_0}$ and the first example in
\ref{ex:trivik}, while $\Gamma _{v_0}^d$ is the connected sum of
$\Gamma _{v_0}$ and the second example of \ref{ex:trivik}.  Since the
master complexes of the two graphs of Example~\ref{ex:trivik}
coincide, we conclude that
\begin{cor}\label{c:egyenlok}
  The master complexes $\MCFinfComb (\Gamma ^+_{v_0})$ and
  $\MCFinfComb (\Gamma ^d_{v_0})$ are equal. In fact, both master complexes
  are equal to $\MCFinfComb (\Gamma _{v_0} )$.
\end{cor}
\begin{proof}
  Both master complexes are the tensor product (over $\Field [\uuinv
  ]$) of the master complex of $\Gamma _{v_0}$ and of $\Field [\uuinv
  ]$, concluding the argument.
\end{proof}
\begin{figure}[ht]
\begin{center}
\epsfig{file=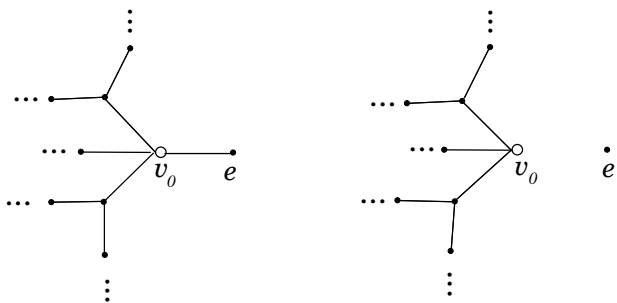, height=4cm}
\end{center}
\caption{{\bf The two graphs $\Gamma _{v_0}^+$ (on the left) and
    $\Gamma _{v_0}^d$ (on the right) derived from a given graph
    $\Gamma _{v_0}$.} The framing of $e$ is $(-1)$ in both cases, and
  $v_0$ is the distinguished vertex (hence admits no framing and is
  denoted by a hollow circle) in both graphs.}
\label{fig:ketgraf}
\end{figure}

\section{Surgery along knots}
\label{sec:surgery}
A formula for computing the lattice homology for the graph
$G_{v_0}$ (we get from $\Gamma _{v_0}$ by attaching appropriate
framing to $v_0$) can be derived from the knowledge of the master 
complex of $\Gamma _{v_0}$, according to the following result:
\begin{thm}\label{thm:mutet}
  The master complex $\MCFinfComb (\Gamma _{v_0})$ of $\Gamma _{v_0}$
  determines the lattice homology of the result of the 
  graph obtained by marking $v_0$ with any integer $m_0\in \Z$,
  for which the resulting graph is negative definite.
\end{thm}

In order to verify this result, first we describe the chain complex
computing lattice homology as a mapping cone of related objects.  As
before, consider the tree $\Gamma _{v_0}$ in which each vertex except
$v_0$ is equipped with a framing. The plumbing graph $G$ is then given
by deleting $v_0$ from $\Gamma _{v_0}$. Let $G_{v_0}=G_{n}(v_0)$
denote the plumbing graph we get from $\Gamma _{v_0}$ by attaching the
framing $n\in \Z$ to $v_0$.  Suppose that for the chosen $n$ the graph
$G_{v_0}$ is negative definite. Our immediate aim is to present the
chain complex $\CFmComb (G_{v_0})$ as a mapping cone of related
objects. These related objects then will be reinterpreted in terms of
the master complex $\MCFinfComb (\Gamma _{v_0})$.

Consider the two-step filtration on $\CFmComb(G_{v_0})$ where the
filtration level of $U^j\otimes [L,H]$ is 1 or 0 according to whether
$v_0$ is in $H$ or $v_0$ is not in $H$. Denoting the elements with
filtration at most 0 by ${\mathbb B}$, we get a short exact sequence
\[
0 \longrightarrow {\mathbb B} \longrightarrow
\CFmComb(G_{v_0})\longrightarrow {\mathbb D} \longrightarrow 0 .
\]
Explicitly, ${\mathbb {B}}$ is generated by pairs $[L,H]$ with
$v_0\not\in H$, while a  nontrivial element in ${\mathbb {D}}$ can be
represented by (linear combinations of) terms $U^j\otimes [L,H]$ where
$v_0\in H$. Indeed, the quotient complex ${\mathbb {D}}$ can be
identified with the complex $({\mathbb {T}}, \partial _{\mathbb
  {T}})$, where ${\mathbb {T}}$ is generated over $\Field [U]$ by
those elements $[L,H]$ of $\Char (G)\times {\mathbb {P}}(V)$ for which
$v_0\in H$, and
\[
{\partial _{\mathbb {T}}[L,H]} = 
\sum_{v\in H-v_0} U^{a_v[L,H]}\otimes [L,H-v] + \sum_{v\in H-v_0}
U^{b_v[L,H]}\otimes [L+2v^*,H-v] .
\]
Notice that there are two obvious maps $\partial _1, \partial _2 
\colon {\mathbb {T}}\to {\mathbb {B}}$:
For a generator $[L,H]$ of ${\mathbb {T}}$ (with $v_0\in H$)
consider 
\begin{equation}\label{eq:hatarok}
\partial _1 [L,H]=U^{a_{v_0}[L,H]}\otimes [L,H-v_0], \qquad \partial _2
[L,H]= U^{b_{v_0}[L,H]} \otimes [L+2v_0^*, H-v_0].
\end{equation}
It follows from $\partial ^2=0$ that both maps 
 $\partial _1 , \partial _2 \colon {\mathbb {T}} \to  {\mathbb {B}}$ 
 are chain maps.
It is easy to see that 
\begin{lem}
  The mapping cone of $({\mathbb {T}}, {\mathbb {B}}, \partial _1
  + \partial _2)$, is chain homotopic to the chain complex $\CFmComb
  (G_{n}(v_0))$ computing the lattice homology $\HFmComb (G_{n}(v_0))$
  of the result of $n$-surgery on $v_0$. \qed
\end{lem}
Next we identify the above terms using the Alexander filtration on
$\CFinfComb (G)$ induced by $v_0$.  We will use the class $\Sigma$ characterized
in Equation~\eqref{eq:DefSigma}.
\begin{defn}
  Consider the subcomplex $B_i\subset {\mathbb B}\subset \CFmComb
  (G_{v_0})$ generated by $[L,H]$ where $\frac{1}{2}(L(\Sigma )+\Sigma
  ^2)= i\in {\mathbb {Q}}$. (Recall that since $[L,H]$ is in ${\mathbb
    {B}}$, the set $H$ does not contain $v_0$. Also, as before, we regard
  $L\in H^2 (X_{G_{v_0}} ; \Z )$ as a cohomology class with rational
  coefficients.)  Since $v_j^*(\Sigma )=v_j\cdot \Sigma =0$ for all
  $j\neq 0$, it follows that $B_i$ is, indeed, a subcomplex of
  ${\mathbb B}$ for any rational $i$, and obviously $\oplus _{i\in
    {\mathbb {Q}}} B_i = {\mathbb {B}}$.
\end{defn}

\begin{prop}\label{p:bperiod}
There is an isomorphism   $\varphi \colon B_i\to B_{i+1}$.
\end{prop}

\begin{proof}
  Define the map $\varphi $ by sending a generator $[L,H]$ of $B_i$ to
  $[L',H]$ where
  \[  L'(v_j)
  = \left\{\begin{array}{ll}
      L(v_0) + 2 & {\text{if $j=0$}} \\
      L(v_j)  & {\text{if $j\neq 0$}}.
      \end{array}
    \right.\] Since $v_0\not \in H$, it follows that
    $f([L,H])=f([L',H])$ (where $f$ is defined in
    Equation~\eqref{eq:fdef}), hence the resulting map is an
    isomorphism between the chain complexes $B_i$ and $B_{i+1}$.
\end{proof}

\begin{prop}\label{prop:izomphi}
The sum $B=\bigoplus_{0\leq i < 1} B_i$ is isomorphic 
to $\CFmComb(G)$. 
\end{prop}

\begin{proof}
  Consider the map $F'\colon B \to \CFmComb (G)$ 
 induced by the forgetful map $F'$ defined
  as $[L,H] \mapsto [L|_{G},H]$.  It is easy to see that (since $H$ does
not contain $v_0$) the map $F'$ is a chain map. Indeed, $F'$ 
    is an isomorphism: one needs to check only that
  every element $[L|_{G},H]$ admits a unique lift to $[L,H]\in B_i$
  with $0\leq i <1$. The condition $\frac{1}{2}(L(\Sigma ) +\Sigma
  ^2)= \frac{1}{2}L(v_0)+\frac{1}{2}(L|_G)(\Sigma -v_0)
  +\frac{1}{2}\Sigma ^2 \in [0,1)$ uniquely characterizes the value of
  $\frac{1}{2}L(v_0)$ by the fact that $L(v_0)\equiv v_0^2 \mod 2$.
\end{proof}
\begin{rem}
Obviously, the same argument shows that 
$\bigoplus _{r\leq i<r+1}B_i$ is isomorphic to $\CFmComb (G)$.
\end{rem}

The above statement admits a spin$^c$-refined version as follows.
Notice first that if we fix a spin$^c$ structure $\t $ on the
3-manifold $Y_{G_{v_0}}$ we get after the surgery and also fix $i$,
then there is a unique spin$^c$ structure $\s$ on $Y_G$ induced by
$(\t, i)$. Indeed, if the cohomology class $L$ satisfies $\s_L=\t$ and
$\frac{1}{2}(L(\Sigma )+\Sigma ^2)=i$, and $L'$ is another
representative of $\t$, then
\[
L'=L +\sum _{i=0}^n 2n_iv_i^*. 
\]
In order for $L'$ to be also in $B_i$, however, the coefficient $n_0$
of $v_0^*$ in the above sum must be equal to zero, hence $L\vert _G$
and $L'\vert _G$ represent the same spin$^c$ structure on $Y_G$. We
will denote this restriction by $(\t , i) \vert _G$.  Then the above
isomorphism $F'$ provides
\begin{lem}
  Let $B_i(\t )$ be the subcomplex of $B_i$ generated by those pairs for
  which $L$ represents the spin$^c$ structure $\t$. The map $F'$
  provides an isomorphism between $B_i(\t)$ and $\CFmComb (G, (\t , i)\vert
  _G)$.
\end{lem}
\begin{proof}
  By the above discussion it is clear that $F'$ maps $B_i(\t )$ to
  $\CFmComb (G, (\t ,i) \vert _G)$. The map is injective, hence to
  show the isomorphism we only need to verify that $F'$ is onto.
  Obviously $L(\Sigma ) +\Sigma ^2=2i$ and $L\vert _G=K$ determines
  $L(v_0)$, and it is not hard to see that for the resulting
  cohomology class $\s _L =\t$.
\end{proof}

In conlcusion, the complexes ${\mathbb {B}}$, $B_i(\t)$ and $B=\oplus
_{i\in [0,1)}B_i$ can be recovered from $\CFmComb (G)$, and hence from 
the master complex.

The complex ${\mathbb {T}}$ also admits a decomposition into $\oplus
_{i\in {\mathbb {Q}}}T_i$ where the generator $[L,H]$ with $v_0\in H$ belongs to
$T_i$ if $\frac{1}{2}(L(\Sigma )+\Sigma ^2) =i\in {\mathbb
  {Q}}$. Notice that the map $\partial _1$ defined in
\eqref{eq:hatarok} maps $T_i$ into $B_i\subset {\mathbb {B}}$, while
when we apply $\partial _2$ to $T_i$, we get a map pointing to
$B_{i+v_0^*(\Sigma)}\subset {\mathbb {B}}$.

Recall that in the definitions of $B_i$ and $T_i$ we used the fixed
framing attached to the vertex $v_0$. In the following we show that
the result will be actually independent of this choice. To formulate
the result, suppose that for the fixed framing $v_0^2n$ the complex
${\mathbb {B}}={\mathbb {B}} (n)$ splits as $\oplus _i B_i (n)$ (and
similarly, ${\mathbb {T}}={\mathbb {T}} (n)$ splits as $\oplus _i T_i
(n)$) .

\begin{lem}
  The chain complexes $B_i (n)$ and $B_{i}(n+1)$ (and similarly $T_i
  (n)$ and $T_{i}(n+1)$) are isomorphic.
\end{lem}
\begin{proof}
  Consider the map $t\colon B_i(n) \to B_{i}(n+1)$ (and similarly
  $t'\colon T_i(n) \to T_{i}(n+1)$) which sends the generator $[L,H]$
  to $[L',H]$ where $L'(v_j)=L(v_j)$ for all $j>0$ and
  $L'(v_0)=L(v_0)-1$.  Notice that by changing the framing on $v_0$
  from $n$ to $n+1$ we increase $\Sigma ^2$ by 1.  Since $L'(\Sigma
  )=L(\Sigma )-1$, and the above map $t$ is invertible, the claim
  follows. Since the function $f$ we used in the definition of the
  boundary map takes the same value for $[L,H]$ as for $[L',H]$, the
  maps $t$ and $t'$ are, indeed, chain maps between the chain complexes.
\end{proof}

Our next goal is to reformulate ${\mathbb {T}}$ (and its splitting as
$\oplus _{i\in {\mathbb {Q}}}T_i$) in terms of the master
complex $\MCFinfComb (\Gamma _{v_0})$.  As before, recall that for a
spin$^c$ structure $\t$ on $Y_{G_{v_0}}$ and $i$ we have a
restricted spin$^c$ structure $\s =(\t, i )\vert _G$ on $Y_G$.
Consider the subcomplex $S_i(\s )=S_i((\t, i )\vert _G) \subset
\CFinfComb (G, \s )$ generated by the elements
\[
\{ U^j\otimes [K,E]\in \CFmComb (G, \s )\mid -j\leq 0, A(U^j\otimes
[K,E])\leq i\}.
\]
\begin{lem}\label{l:izom}
  For a spin$^c$ structure $\t$ the chain complex $T_i(\t )$ and
  the subcomplex $S_i ((\t , i) \vert _G)$ are isomorphic as chain
  complexes.
\end{lem}
\begin{proof} 
  Define the map $F=F^{\t }_i\colon T_i (\t) \to S_i((\t , i)\vert
  _G)$ on the generator $[L,H]$ by the formula
\[
F([L,H])=U^{a_{v_0} [L,H]}\otimes [L\vert _G,H-v_0].
\]
The exponent of $U$ in this expression is obviously nonnegative and
the spin$^c$ structure of the image is equal to $(\t , i)\vert _G$.
Therefore, in order to show that $F([L,H])\in S_i((\t , i) \vert _G)$, we
need only to verify that
\begin{equation}\label{eq:kisebb}
A(F([L,H])) \leq i=\frac{1}{2}(L(\Sigma )+\Sigma ^2).
\end{equation}
In fact, we claim that 
\begin{equation}\label{eq:relab}
  \frac{1}{2}(L(\Sigma )+\Sigma ^2)-  A(U^{a_{v_0}([L,H])}\otimes 
  [L\vert _G,H-v_0])=  b_{v_0}[L,H].
\end{equation}
By substituting the definitions of the various terms in the left hand
side of this equation (after multiplying it by 2), and applying the
obvious simplifications we get
\[
L(v_0)+2g([L,H-v_0]-2g([L, H]) 
+v_0^2-2g([L\vert _G, H-v_0])+2g([L\vert _G+2v_0^*, H-v_0]).
\]
Since $g([L\vert _G, H-v_0])=g([L,H-v_0])$, this expression is clearly equal
to $2b_{v_0}[L,H]$, concluding the argument. Since $b_{v_0}[L,H]$ is nonnegative,
Equation~\eqref{eq:relab} immediately implies the inequality of~\eqref{eq:kisebb}.

Finally, a simple argument shows that $F$ is a chain map:
The two necessary identities
\[
a_{v_0}[L,H]+a_v[L\vert _G, H-v_0]=a_v[L,H]+a_{v_0}[L,H-v]
\]
and
\[
a_{v_0}[L,H]+b_v[L\vert _G, H-v_0]=b_v[L,H]+a_{v_0}[L+2v^*, H-v]
\]
are reformulations of Equations~\eqref{eq:olto} and~\eqref{eq:olto2}
(together with the observation that $f(L\vert _G,
I)=f(L,I)$ once $v_0\not \in I$).

Next we show that $F$ is an isomorphism. For $[K,E]$ on $G$ there is a
unique extension $[L,H]$ on $G_{v_0}$ with $[L\vert _G, H-v_0]=[K,E]$
and $\frac{1}{2}(L(\Sigma ) +\Sigma ^2)=i$, hence the injectivity of
$F$ easily follows. To show that $F$ is onto, fix an element
$U^j\otimes [K,E]\in S_i ((\t , i)\vert _G )$ and consider $[L,H]\in
T_i (\t )$ with $F([L,H])=U^{a_{v_0}[L,H]}\otimes [K,E]$.  If
$a_{v_0}[L,H]=0$ then $U^j\otimes [L,H]$ maps to $U^j\otimes [K,E]$
under $F$. In case $a_{v_0}[L,H]>0$ then $b_{v_0}[L,H]=0$ and so by
the identity of \eqref{eq:relab} we get that
$A(U^{a_{v_0}[L,H]}\otimes [K,E])=i$.  Therefore $A(U^j\otimes
[K,E])\leq i$ implies that $j\geq a_{v_0}[L,H]$, hence
$U^{j-a_{v_0}[L,H]}\otimes [L,H]$ is in $T_i(\t )$ and maps under $F$
to $U^j \otimes [K,E]$, concluding the proof.
\end{proof}

The subcomplexes of ${\mathbb {T}}$ admit a certain symmetry, induced by 
the $J$-map.
\begin{lem}
The $J$-map induces an isomorphism $J_i$ between the chain
complexes $T_i$ and $T_{-i}$. This isomorphism intertwines the
maps $\partial _1$ and $\partial _2$; more precisely
$\partial _2$ on $T_i$ is equal to $J_i^{-1}\circ \partial _1 \circ J_i$
(and $\partial _1$ on $T_i$ is equal to $J_i^{-1}\circ \partial _2 \circ J_i$). 
\end{lem}
\begin{proof}
  Recall the definition $J[L,H]=[-L-\sum _{v\in H}2v^*, H]$ of the
  $J$-map on the chain complex $\CFmComb (G _{v_0})$. Applying it to the
  complex $T_i$, we claim that we get a chain complex isomorphism
  $J_i\colon T_i\to T_{-i}$: from the fact $(-L-\sum _{v\in H}2v^*)(\Sigma )=
  -L(\Sigma )-2v_0\cdot \Sigma$ (since $v_0\in H$ and for all other
  $v_i$ we have that $v_i\cdot \Sigma =0$) together with the
  observation that $\Sigma ^2 = v_0\cdot \Sigma$, it follows that
\[
\frac{1}{2}((-L-\sum _{v\in H}2v^*)(\Sigma) +\Sigma ^2)=\frac{1}{2}(-L(\Sigma )-
\Sigma ^2 )=-\frac{1}{2}(L(\Sigma ) + \Sigma ^2).
\]
This equation shows that $J_i$ maps $T_i$ to $T_{-i}$. The claim
$\partial _2 = J_i^{-1}\circ \partial _1 \circ J_i$ (where $\partial
_2$ is taken on $T_i$ while $\partial _1$ on $T_{-i}$) then simply
follows from the identities of~\eqref{eq:spinek} in
Lemma~\ref{l:spinc}.
\end{proof}

The same idea as above shows that 
\begin{lem}
The restriction of $J$ to $B_i$ provides an isomorphism 
$B_i\to B_{-i+v_0^*(\Sigma )}$ of chain complexes.
\end{lem}
\begin{proof} 
Indeed, if $v_0\not \in H$, then $(-L-\sum _{v\in H}2v^*)(\Sigma )= -L(\Sigma)$, hence
\[
\frac{1}{2}((-L-\sum _{v\in H}2v^*)(\Sigma) +\Sigma ^2)=\frac{1}{2}(-L(\Sigma )+
\Sigma ^2 )=-\frac{1}{2}(L(\Sigma ) + \Sigma ^2)+\Sigma ^2,
\]
and $\Sigma ^2=v_0^*(\Sigma )$.
\end{proof}

Next we identify the two maps $\partial _1$ and $\partial _2$ of the
mapping cone $({\mathbb {T}}, {\mathbb {B}}, \partial _1 + \partial
_2)$ in the filtered lattice chain complex context. Notice that $S_i
(\s )$ is naturally a subcomplex of $\CFmComb (G, \s )$; let the
inclusion $S_i(\s )\subset \CFmComb (G, \s )$ be denoted by $\eta
_1$. It is obvious from the definitions that for the maps $F',F$ of
Proposition~\ref{prop:izomphi} and Lemma~\ref{l:izom}
\[
F'(\partial _1 [L,H])=\eta _1 (F ([L,H])).
\]

The subcomplex $S_i (\s )$ admits a further natural embedding into the
complex $V_i(\s )$ which is generated in $\CFinfComb (G, \s )$ by the
elements $\{ U^j \otimes [K,E]\mid A (U^j \otimes [K,E])\leq i\}$.
($V_i (\s )$ is the subcomplex of $\CFinfComb (G, \s )$ when we regard
this latter as an $\Field [U]$-module.) Recall that $s_{v_0}$ denotes
the spin$^c$ structure we get from $\s$ by twisting it with $v_0^*$.
\begin{prop}
The subcomplex $V_i(\s )$ is isomorphic to $\CFmComb (G, \s_{v_0})$.
\end{prop}
\begin{proof}
Consider the map $U^{i-i_{\s}}N$ from Definition~\ref{def:mapn}
mapping from $\CFinfComb (G, \s )$ to $\CFinfComb (G, \s _{v_0})$. It
is easy to see that this map provides an isomorphism between $V_i (\s
)$ and $\CFmComb (G, \s _{v_0})$, since
\[
j(U^{i-i_{\s}}\otimes N(U^k \otimes [K,E]))=i+k -A([K,E])
\]
is nonnegative if and only if $i\geq -k+A([K,E])=A(U^k \otimes [K,E])$.
\end{proof}

Define now $\eta _2\colon S_i (\s )\to \CFmComb (G, \s _{v_0})$ as the 
composition of the embedding $S_i (\s )\to V_i (\s )$ with the map
$U^{i-i_{\s }}N$. With this definition in place the identity
\[
\eta _2 \circ F = F'\circ \partial _2
\]
easily follows: 
\[
(\eta _2 \circ F)[L,H]=U^{a_{v_0}[L,H]+i -A([L\vert _G , H-v_0])}\otimes [L\vert _G +2v_0^*, H-v_0], 
\]
\[
(F'\circ \partial _2 )[L,H]= U^{b_{v_0}[L,H]}[L+2v_0^*\vert _G ,
H-v_0],
\]
and the two right-hand-side terms are equal by the identity of
\eqref{eq:relab}. Now we are in the position to turn to the proof of
the main result of this section, Theorem~\ref{thm:mutet}.

\begin{proof}[Proof of Theorem~\ref{thm:mutet}]
  Fix the framing $n$ of $v_0$ in such a way that $G_{v_0}=G_{n}(v_0)$
  is a negative definite plumbing graph. Fix a spin$^c$ structure $\t$
  on $Y_{G_{v_o}}$. Our goal is now to determine the chain complex
  $\CFmComb (G_{v_0}, \t )$ from the master complex of $\Gamma
  _{v_0}$. As we discussed earlier in this section, it is sufficient
  to recover the subcomplexes $T_i (\t )$, $B_i(\t)$ (for $i\in \{
  q+n\cdot \Sigma ^2 \mid n\in {\mathbb {N}}\}$ for an appropriate
  $q\in {\mathbb {Q}}$) and the maps $\partial _1\colon T_i (\t )\to
  B_i (\t )$ and $\partial _2 \colon T_i (\t ) \to B_{i+v_0^*(\Sigma
    )}(\t )$.

Identify $T_i (\t )$ with the subcomplex $S_i ((\t , i )\vert _G)$ and 
$B_i (\t )$ with $\CFmComb (G, (\t , i )\vert _G )$ (both as 
subcomplexes of $\CFinfComb (G, (\t , i )\vert _G)$) by the maps
$F$ and $F'$. As we showed earlier, the natural embedding of
$S_i ((\t , i )\vert _G)\subset \CFmComb (G, (\t , i)\vert _G)$ can play
the role of $\partial _1$, while the embedding $S_i((\t , i )\vert _G)
\to V_i((\t , i )\vert _G)$ composed with $U^{i-i_{(\t , i )\vert _G}}N$
plays the role of $\partial _2$ in this model. These subcomplexes
and maps are all determined by $\CFinfComb (G)$, the two filtrations
and the map $N$ on it. Since by its definition the master complex
of $\Gamma  _{v_0}$ equals this collection of data, the theorem is
proved.
\end{proof}

\subsection{Computation of the master complex}
When computing the homology $\HFmComb (G_{n}(v_0))$ from $(\oplus S_i,
   \oplus _{k\in \Z} \CFmComb (G), \eta_1 , \eta _2)$ we can first
take the homologies $H_*(S _i)$ and $\HFmComb (G)$ and
consider the maps $H_*(\eta _1)$ and $H_*(\eta _2)$ induced by
$\eta_1, \eta _2$ on these smaller complexes.  This method provides
more manageable chain complexes to work with, but it also loses some
information: the resulting homology will be isomorphic to the homology
of the original mapping cone only as a vector space over $\Field$, and
not necessarily as a module over the ring $\Field [U]$.  Nevertheless,
sometimes this partial information can be applied very conveniently.

As an example, we show how to recover (in favorable
situations) the knot lattice homology $\HFKaComb (\Gamma _{v_0} )$
from the homologies of $S_i$. Let us consider the following iterated
mapping cone.  First consider the mapping cones $C_i$ of $ (S_i,
S_{i+1}, \psi _i)$ for $i=n,n-1$, and then consider the mapping cone
$D(n)$ of $(C_n,C_{n-1}, (\phi _{i+1},\phi _i))$.  (For a schematic
picture of the chain complex, see Figure~\ref{fig:mapcone}.)  In the
next lemma we will still need to use the complexes $S_i$ rather than
their homologies.
\begin{figure}[ht]
\begin{center}
\epsfig{file=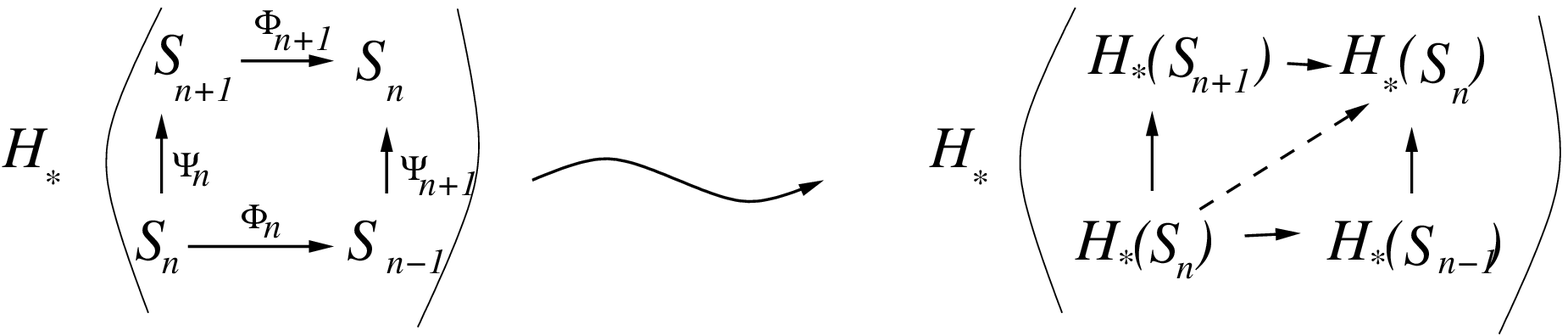, height=2.6cm}
\end{center}
\caption{{\bf The iterated mapping cone $D(n)$ on the $S_i$'s.}  The
  maps are defined as $\phi _i, \psi _i$ with appropriate choices of
  $i$ on the left, and the homomorphisms induced by these maps on the
  right. When taking homologies first, we might need to encounter a
  nontrivial map indicated by the dashed arrow.}
\label{fig:mapcone}
\end{figure}
\begin{lem}\label{l:comput}
  The homology $H_*(D(n))$ is isomorphic to $\HFKaComb (\Gamma _{v_0},
  n)$.
\end{lem}
\begin{proof}
  Factoring $S_{n+1}$ with the image of $\psi _n\colon
  S_n\to S_{n+1}$ we compute the homology of the horizontal strip
  in the master complex with $A=n+1$ and nonnegative $U$-power (i.e.,
  $j\geq 0$).  Similarly, with the help of $\psi _{n-1}\colon
  S_{n-1}\to S_n$ we get the homology of the horizontal strip with
  $A=n$ and nonnegative $U$-power. The iterated
  mapping cone in the statement maps the upper strip into the lower one by
  multiplying it by $U$, localizing
  the computation to one coordinate with $A=n$ and vanishing
  $U$-power. The homology of this complex is by definition the knot
  lattice homology $\HFKaComb (\Gamma _{v_0},n)$. 
  \end{proof}

Unfortunately, if we first take the homologies of the complexes $S_i$
and then form the mapping cones in the above discussion, we might get
different homology. The reason is that when taking homologies of the
$S_i$ we might need to consider a diagonal map, as indicated by the dashed arrow
of Figure~\ref{fig:mapcone}. Under favorable circumstances, however,
the diagonal map can be determined to be zero, and in those cases
$\HFKaComb (\Gamma _{v_0})$ can be computed from the homologies of
$S_i$ (and the maps induced by $\phi _i, \psi _i$ on these
homologies). From the knowledge of $\HFKaComb (\Gamma _{v_0}, n)$ we
can recover the nontrivial groups in the master complex:
multiplication by $U^n$ simply translates $\HFKaComb (\Gamma _{v_0})$
(located on the $y$-axis) with the vectors $(n,n)$ ($n\in \Z $).  In
some special cases appropriate \emph{ad hoc} arguments help us to
reconstruct the differentials and the map $N$ on the master complex
(which do not follow from the computation of $\HFKaComb (\Gamma
_{v_0})$), getting $\MCFinfComb (\Gamma _{v_0})$ back from $H_* (S_i)$
and the maps $H_*(\Psi _i)$ and $H_*(\Phi _i)$.

Remember also that first taking the homology and then the mapping cone
causes some information loss: the result will coincide with the
homology of the mapping cone as a vector space over $\Field$, but not
necessarily as an $\Field [U]$-module. The vector space underlying the
$\Field [U]$-module $\HFmComb$ is already an interesting invariant of
the graph. The module structure can be reconstructed by considering
the mapping cones with coefficient rings $\Field [U]/(U^n)$ for every
$n\in {\mathbb {N}}$, cf. \cite[Lemma~4.12]{OSSZlatt}.

\section{An example: the right-handed trefoil knot}
\label{sec:trefoil}
In this section we give an explicit computation of the filtered
lattice chain complex (introduced in Section~\ref{sec:knots}) for the
right-handed trefoil knot in $S^3$. It is a standard fact that this
knot can be given by the plumbing diagram $\Gamma _{v_0}$ of
Figure~\ref{fig:trefoilplumb}. Notice that in this example the
background manifold is diffeomorphic to $S^3$, hence admits a unique
spin$^c$ structure, and therefore we do not need to record it.
(Related explicit computations can be found in \cite{s3csomok}.)
 
\begin{figure}[ht]
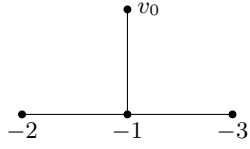

\begin{center}
\setlength{\unitlength}{1mm} \unitlength=0.7cm
\begin{graph}(5,2)(-2,0)
\graphnodesize{0.15}

\roundnode{n1}(-1,0)
\roundnode{n2}(1,0)
\roundnode{n3}(3,0)
\roundnode{n4}(1,2)

\edge{n1}{n2}
\edge{n2}{n3}
\edge{n2}{n4}

\autonodetext{n1}[s]{\small $-2$} 
\autonodetext{n2}[s]{\small $-1$} 
\autonodetext{n3}[s]{\small $-3$}
\autonodetext{n4}[e]{\small $v_0$}

\end{graph}
\end{center}
\caption{{\bf The plumbing tree $\Gamma _{v_0}$ describing the
    right-handed trefoil knot in $S^3$.}  Interpreting the graph as a
  plumbing tree, the repeated blow-down of the $(-1)$-, $(-2)$- and
  $(-3)$-framed vertices turn the circle corresponding to $v_0$ into
  the right-handed trefoil knot.}
\label{fig:trefoilplumb}
\end{figure}
Using the results of \cite{nemethi-ar, lattice} first we will
determine $H_*(T_i)$ and $H_*(B)$ when the framing $v_0^2=-7$ is fixed
on $v_0$.  
\begin{prop}
  Suppose that $\Gamma _{v_0}$ is given by the diagram of
  Figure~\ref{fig:trefoilplumb}. Then $H_*(B)\cong \Field [U]$.
\end{prop}
\begin{proof}
  The graph $G=\Gamma _{v_0}-v_0$ is negative definite with one bad
  vertex, hence the result of \cite{lattice} (cf. also
  \cite{nemethi-ar}) applies and shows that the lattice homology of it
  is isomorphic to the Heegaard Floer homology of the 3-manifold $Y_G$
  defined by the plumbing. Since $G$ presents $S^3$ as a 3-manifold
  and $H_*(B)\cong \HFmComb (G)$, the claim follows.
\end{proof}
Consequently the lattice homology group $\HFmComb (G) =\HFmComb
_0(G)\cong H_*(B)$ is generated by a single element, and it has to be
a linear combination of elements of the form $[K,E]$ with
$E=\emptyset$ (since the entire homology of a negative definite graph
with at most one bad vertex is supported in this level). The generator
has Maslov grading 0, which by the definition of the grading means
that $\frac{1}{4}(K^2+3)=0$, i.e. $K^2=-3$. There are exactly $8$ such
cohomology classes on $G$, and it is easy to verify that these
are all homologous to each other (when thought of as cycles in lattice
homology), so any one of them can represent the generator of $\HFmComb
(G)=\Field [U]$. By denoting the vertex of $G$ with framing $-i$ by
$v_i$ ($i=1,2,3$), we define the vector $K$  as
\begin{equation}\label{eq:defk}
(K(v_1), K(v_2), K(v_3))=(-1,0,1) .
\end{equation}
Simple calculation shows that $K^2=-3$, hence $[K, \emptyset ]$
generates $\HFmComb (G)$. We will need one further computational
fact for the group $\HFmComb (G)$:
\begin{lem}\label{l:kisszamolas}
The element $[K', \emptyset ]\in \CFmComb (G)$
given by $(K'(v_1), K'(v_2), K'(v_3))=(1,0,1) $
is homologous to $U\otimes [K, \emptyset ]$, where 
$K$ is given by \eqref{eq:defk} above.
\end{lem}
\begin{proof}
Consider the element
\[
x=[(1,0,1), \{ v_1\} ]+[(-1,2,3), \{v_3\} ]+[(1,2,-3), \{ v_1\} ]+
[(-1,4,-1), \{v_2\} ] .
\]
It is an easy computation to show that 
$\partial x = [(1,0,1), \emptyset ]+U\otimes [(1,0,-1), \emptyset]$.
Since  both $[K, \emptyset ]$ and $[(1,0,-1), \emptyset ]$
generate $\HFmComb (G)$, the proof is
complete. 
\end{proof}

Before calculating $H_* (T_i)$, we determine the maps $H_*(\partial
_1), H_*(\partial _2 )\colon H_*(T_i) \to H_*(B)$ on certain elements.  To
this end, for $j\in \Z$ consider the elements $L_j\in H^2(X_{G_{v_0}};
\Z )$ (with framing $v_0^2=-7$ attached to $v_0$) defined as
\[
(L_j(v_1), L_j(v_2), L_j(v_3), L_j(v_0))=( -1,0,1,2j+1).
\]
Since $\Sigma =v_0+6v_1+3v_2+2v_3$, by the choice $v_0^2=-7$ we get
$\Sigma ^2=-1$. This implies that $\frac{1}{2}(L_j(\Sigma ) +\Sigma
^2)=j-2$, hence the element $[L_j, \{ v_0\}]$ is in $T_{j-2}$. Simple
calculation shows that
  \[  a_{v_0}[L_j, \{ v_0\} ]
  = \left\{\begin{array}{ll}
      0 & {\text{if $j-3\geq 0$}} \\
      -(j-3)  & {\text{if $j-3<0$}}.
      \end{array}
    \right.\]
\[  b_{v_0}[L_j, \{ v_0\} ]
  = \left\{\begin{array}{ll}
      j-3 & {\text{if $j-3\geq 0$}} \\
      0  & {\text{if $j-3<0$}}.
      \end{array}
    \right.\]

With notations $a_j=a_{v_0}([L_j, \{ v_0\}])$  and $b_j=b_{v_0}([L_j, \{ v_0\}])$
we conclude that (with the conventions for $K$ and $K'$ above, and with the
identification of $B$ with $\CFmComb (G)$)
\[
\partial _1 [L_j, \{ v_0\}]=U^{a_j} \otimes K \quad {\mbox {and}}\quad
\partial _2 [L_j, \{ v_0\}]=U^{b_j} \otimes K',
\]
and the latter element (according to Lemma~\ref{l:kisszamolas}) is
homologous to $U^{b_j+1}\otimes K$.  This shows that for $j\geq 3$ the
homology class of $H_*(T_{j-2})$ represented by the element $[L_j, \{
v_0 \}]$ maps under $(\partial _1,
\partial _2)$ to $((-1,0,1), U^{j-2}\otimes (-1,0,1))\in \HFmComb (G)\times
\HFmComb (G)$.
Applying the $J$-symmetry we can then determine the
$(\partial_1, \partial _2)$-image of $J[L_j, \{ v_0\}]\in T_{2-j}$
($j\geq 3$) as well.  (Notice that although $J[L_j, \{ v_0\}]$ and
$[L_{-j+4}, \{ v_0\}]$ are both elements of $T_{-(j-2)}$, they are not
necessarily homologous.) For $j=2$ the class $[L_2,\{ v_0\}]\in T_0$
maps to $(U\otimes (-1,0,1), U\otimes (-1,0,1))$. Now we are in the
position to determine the homologies $H_*(T_i)$, as well as the maps
on them. Notice first that since $G$ represents $S^3$, the Alexander
gradings are all integer valued, hence we have a nontrivial complex
$T_i$ for each $i\in \Z$.

\begin{prop}
The homology $H_*(T_i)$ is isomorphic to $\Field [U]$.
\end{prop}
\begin{proof}
  Notice first that $H_*(T_i)$ cannot have any nontrivial $U$-torsion:
  since $\partial _1, \partial _2$ map to $H_*(B)=\Field [U]$, such
  part of the homology stays in the kernel of $\partial _1$ and
  $\partial _2$, hence would give nontrivial homology in $\HFmComb
  _1(G_{v_0})$ (supported in $\vert E \vert =1$). This, however,
  contradicts the fact that for negative definite graphs with at most
  one bad vertex we have that $\HFmComb _1 (G_{v_0})=0$ \cite{OSzplum,
    lattice}.  If $i> 0$ and $H_*(T_i)$ is not cyclic, then (by the
  $J$-symmetry) the same applies to $H_*(T_{-i})$. Consider the
  surgery coefficient $n$ with the property that $\partial _2$ on
  $T_i$ and $\partial _1$ on $T_{-i}$ point to the same $B$. Then
  $H_*(T_i)\oplus H_*(T_{-i})\to H_*(B)\oplus H_*(B)\oplus H_*(B)$
  will have nontrivial kernel, once again producing nontrivial
  elements in $\HFmComb _1(G_{n}(v_0))$, a group which vanishes for any
  (negative enough) surgery on $v_0$. For the same reason, $H_*(T_0)$
  can have at most two generators, and if it has two generators, then
  the two maps $\partial _1$ and $\partial _2$ have different elements
  in their kernel.  Suppose that $H_*(T_0)$ is not cyclic.  In this
  case (for the choice $v_0^2=-7$) the $U=1$ homology can be easily
  computed and shown to be zero, contradicting the fact that in the
  single spin$^c$ structure on $Y_{G_{-7}(v_0)}$ this homology
  is equal to $\Field$.  This last argument then implies that $H_*
  (T_0)=\Field [U]$ and concludes the proof of the proposition.
\end{proof}

Now our earlier computations of the maps show that for $i>0$ the map
$\partial _1$ maps $[L_{i+2}, \{ v_0\}]\in T_{i}$ into the generator
of $\HFmComb (G)$, hence $[L_{i+2}, \{ v_0\}]$ generates $H_*(T_i)$.
Furthermore, this reasoning shows that $\partial _1$ is an isomorphism
and the map $\partial _2\colon H_* (T_i)\to \HFmComb (G)$ is
multiplication by $U^i$.  By the $J$-symmetry this computation also
determines the maps $\partial _1, \partial _2$ on all $H_*(T_i)$ with
$i\neq 0$.  On $T_0$ the situation is slightly more complicated: both
maps $\partial _1, \partial _2$ take $[L_2, \{ v_0\}]$ to $U$-times
the generator of $\HFmComb (G)$.  This can happen in two ways. Either
$[L_2, \{ v_0\}]$ generates $H_*(T_0)$ (and the maps $\partial
_1, \partial _2$ are both multiplications by $U$), or the cycle
$[L_2, \{ v_0\}]$ is homologous to one of the form $U\otimes g$, where $g$ an be 
represented by a sum of generators (of the form $[L',\{v_0\}]$), each of Maslov grading two greater
than the Maslov grading of $[L_2,\{v_0\}]$. Thus, our aim is to show that
there are no generators in the requisite Maslov grading.

Specifically, we have that
\[ {\grad}[L_2,\{v_0\}]=-1;\]
while
\[ \grad[K,\{v_0\}]=2g[K,\{v_0\}]+1+\frac{1}{4}(K^2+4);\]
which in turn can be $1$ only if $K^2=-4$ and $g[K,\{v_0\}]=0$;
$K^2=-4$
implies that $K(v_0)\leq 5$, while $g[K,\{v_0\}]=0$
implies that $K( v_0)\geq 7$, a contradiction.

We have thus identified the mapping cone $(\oplus _i H_*(T_i), \oplus _{k\in
  \Z}H_*(B), H_*(\partial _1 +\partial _2))$.  For a schematic picture
of the maps, see Figure~\ref{fig:eredmeny}.
\begin{figure}[ht]
\begin{center}
\epsfig{file=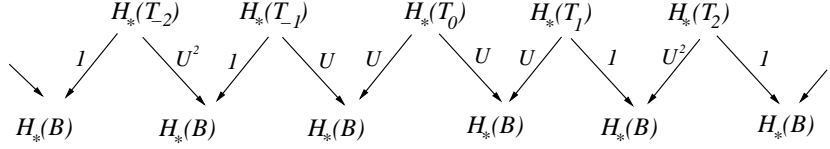, height=2cm}
\end{center}
\caption{{\bf The schematic diagram of the homology groups of
    $H_*(T_i)$, of $H_*(B)$ and the maps between them.} All homologies
  are isomorphic to $\Field [U]$, and the maps are all multiplication
  by some power of $U$ (as indicated in the diagram). The sequence of
  homologies continue in both directions to $\pm \infty$.}
\label{fig:eredmeny}
\end{figure}

We are now ready to describe the master complex of
$\Gamma _{v_0}$. We start by determining the groups on the line
$j=0$ --- equivalently, we compute $\HFKaComb (\Gamma _{v_0})$.  For
this computation, the formula of Lemma~\ref{l:comput} turns out to be
rather useful. Indeed, since $H_*(T_i)=\Field [U]$, there is no
diagonal map in the mapping cone of Figure~\ref{fig:mapcone}. 

The map $H_*(\Psi _i)\colon H_*(T_i)\to H_*(T_{i+1})$ can be
determined from the fact that composing it with the map
$H_*(T_{i+1})\to H_*(B)$ we get $H_*(T_i)\to H_*(B)$. Since $\partial
_1 \colon H_*(T_i)\to H_*(B)$ is an isomorphism for $i\geq 1$, so are
all the maps $H_*(\Psi _i)$. Using the same principle for $i=0$ (and
noticing that $H_*(T_0)\to H_*(B)$ is multiplication by $U$) we get
that $H_*(\Psi _0)$ is also multiplication by $U$. Repeating the same
argument it follows that $H_*(\Psi _{-1})$ is an isomorphism, while
$H_*(\Psi _i)$ is multiplication by $U$ for all $i\leq -2$. The
iterated mapping cone construction of Lemma~\ref{l:comput} shows that
the group $\HFKaComb (\Gamma _{v_0}, n)$ vanishes if the two maps
$H_*(\Psi _{n})$ and $H_*(\Psi _{n-1})$ are the same, and the group
$\HFKaComb (\Gamma _{v_0}, n)$ is isomorphic to $\Field$ is the two
maps above differ. (For similar computations see \cite{OSSzlspace}.)
The computation of the maps $H_*(\Psi _i)$ above shows that
\begin{lem}
For $\Gamma _{v_0}$ given by Figure~\ref{fig:trefoilplumb}
the knot lattice group 
$\HFKaComb (\Gamma _{v_0}, n)$ is isomorphic to $\Field$ for $n=-1,0,1$
and vanishes otherwise. \qed
\end{lem}

Indeed, with the convention used in Equation~\ref{eq:defk}, the group
$\HFKaComb (\Gamma _{v_0}, 1)$ 
can be represented by
\[
x_1=[(-1,0,1), \emptyset], 
\]
while the group 
$\HFKaComb (\Gamma _{v_0}, -1)$ 
by 
\[
x_{-1}=[(-1,0,-1), \emptyset]. 
\]
It is straightforward to determine the Alexander gradings
of these elements, and requires only a little more work to show that these two 
generators are not boundaries of elements of the same Alexander grading.
A quick computation gives that the Maslov grading of $x_1$ is 0,
while the Maslov grading of $x_{-1}$ is $-2$. Since the homology
of the elements with $j=0$ gives $\Field$ in Maslov grading 0 (as the $\HFaComb$-invariant of 
$S^3$), we conclude that the generator $x_0$ of the group 
$\HFKaComb (\Gamma _{v_0}, 0)=\Field $ must be of Maslov grading $-1$.
Furthermore, $x_{-1}$ is one of the components of $\partial x_0$.

Similarly, since the homology along the line $A=0$ is also $\Field$ (supported in 
Maslov grading 0), it is generated by $U^{-1}\otimes x_{-1}$ and therefore there is a nontrivial
map from $x_0$ to $U\otimes x_1$. 
Furthermore, this picture is 
translated by multiplications by all powers of $U$, providing 
nontrivial maps on the master complex. There is no more nontrivial map
by simple Maslov grading argument.  
The filtered chain  complex $\CFinfComb (\Gamma _{v_0})$ is
then described by Figure~\ref{fig:master}. (By convention, a solid dot
symbolizes $\Field$, while an arrow stands for a nontrivial map
between the two 1-dimensional vector spaces.)
\begin{figure}[ht]
\begin{center}
\epsfig{file=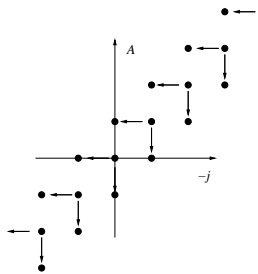, height=3.5cm}
\end{center}
\caption{{\bf The schematic diagram of the master complex
$\MCFinfComb (\Gamma _{v_0})$.}  As usual, nontrivial groups 
are denoted by dots, while nontrivial maps between them are
symbolized by arrows.}
\label{fig:master}
\end{figure}
Furthermore, as the map $N$ is $U$-equivariant, it is equal to the identity.  
Comparing this result with \cite{OSzint} we get that

\begin{prop}\label{prop:mastcomp}
  The master complex of $\Gamma _{v_0}$ determined above is filtered
  chain homotopic to the master complex of the right-handed trefoil
  knot in Heegaard Floer homology (as it is given in
  \cite{OSzknot}). Consequently the filtered lattice chain complex of
  the right-handed trefoil (given by Figure~\ref{fig:trefoilplumb}) is
  filtered chain homotopy equivalent to the filtered knot Floer chain
  complex of the same knot. \qed
\end{prop}

\begin{rems}\label{rem:masszam}
\begin{itemize}  
\item
Essentially the same argument extends to the family of graphs $\{
  \Gamma _{v_0}(n) \mid n\in {\mathbb {N}}\}$ we get by modifying the
  graph $\Gamma _{v_0}$ of Figure~\ref{fig:trefoilplumb} by attaching
  a string of $(n-1)$ vertices, each with framing $(-2)$ to the
  $(-3)$-framed vertex of $\Gamma _{v_0}$.  The resulting knot can be
  easily shown to be the $(2,2n+1)$ torus knot. A straightforward adaptation of
  the argument above provides an identifications of the filtered chain
  homotopy types of the master complexes (in lattice homology) of
  these knots with the master complexes in knot Floer homology.

\item An even simpler computation along the same lines provides the
  master complex of the graph $\Gamma ^k_{v_0}$ introduced in
  Remark~\ref{rem:gkknot}: the complex is isomorphic to $\Field
  [\uuinv ]$, and the Alexander grading of $U^j$ is simply $j$. (This
  computation should not be suprising at all: the knots given by
  $\Gamma ^k_{v_0}$ are all unknots in $S^3$.) In the computation of
  the master complex $\MCFinfComb (\Gamma^k _{v_0})$ of $\Gamma
  ^k_{v_0}$ the graphs after surgery on $v_0$ are all linear, with the
  single $(-1)$ as bad vertex, and so the lattice homologies
  $\HFinfComb (\Gamma ^k_{v_0})$ are isomorphic to $\Field [\uuinv ]$
  (in the unique spin$^c$ structure).  Obviously, since the background
  3-manifold in all the above examples is $S^3$, the map $N$ on
  $\CFinfComb (G)$ must be the identity.
\end{itemize}
\end{rems}

As an application, consider the connected sum of $n$ trefoil
knots. (For a plumbing diagram, see Figure~\ref{fig:connsum}.)

\begin{proof}[Proof of Theorem~\ref{thm:pelda}]
  According to Proposition~\ref{prop:mastcomp},
  together with
  the connected sum formula for lattice homology and the K\"unneth
  formula for knot Floer homology, we get that the two filtered chain
  complexes for $v_0$ in Figure~\ref{fig:connsum} (the filtered
  lattice chain complex and the knot Floer chain complex) are filtered
  chain homotopic to each other. (See Figure~\ref{fig:nketto} for the
  master complex we get in the $n=2$ case.)
\begin{figure}[ht]
\begin{center}
\epsfig{file=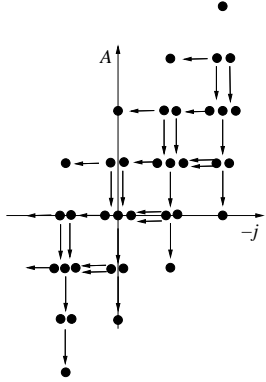, height=5cm}
\end{center}
\caption{{\bf The  master complex for the knot $T\# T$ (where $T$ is the
right-handed trefoil knot).}}
\label{fig:nketto}
\end{figure}
Equip the vertex $v_0$ of Figure~\ref{fig:connsum} with framing
$m_0\leq -6n-1$. Then the corresponding 3-manifold is
$(m_0+6n)$-surgery on the $n$-fold connected sum of trefoil knots in
$S^3$. Since the master complex determines the chain complex of the
surgery in the same manner in the two theories, the lattice homology
of this graph is isomorphic to the Heegaard Floer homology of the
corresponding 3-manifold.
\end{proof}
\begin{rem}
  Notice that this graph has exactly $n$ bad vertices, therefore the
  above result provides further evidence to the conjectured
  isomorphism of lattice and Heegaard Floer homologies. (For related
  results also see \cite{s3csomok}.)
More generally, the identification of the master complexes of 
knots in $S^3$ (in fact in any $Y_G$ which is an $L$-space) is given in 
\cite{OSSzlspace}.
\end{rem}

\section{Appendix: The proof of invariance}
\label{sec:app}
Using the filtered chain complexes for various graphs, in this
Appendix we will give a proof of the result of N\'emethi quoted in
Theorem~\ref{thm:indep}.
According to a classical result of Neumann~\cite{neumann}, the two
3-manifolds $Y_{G_1}$ and $Y_{G_2}$ 
(associated to negative definite plumbing forests $G_1, G_2$)
are diffeomorphic if and only if
the plumbing forests $G_1$ and $G_2$ can be connected by a finite
sequence of blow-ups and blow-downs. Note that since $G_1,G_2$ are
trees/forests, there are three types of blow-ups:
\begin{itemize}
\item we can take the disjoint union of our graph with the graph with
  a single $(-1)$-framed vertex $e$, with no edges emanating from it;
\item we can blow up a vertex $v$, introducing a new leaf $e$ of the graph
  with framing $(-1)$ connected only to $v$, while dropping the framing of  $v$ by one, or
\item we can blow-up an edge connecting vertices $v_1, v_2$, where the
  new vertex $e$ will have valency two and framing $(-1)$, while the
  framings of $v_1,v_2$  will drop by one. Also, $v_1$ and $v_2$ 
are no longer connected, but both are connected to $e$.
\end{itemize}
Correspondingly, we can blow down only those vertices, which have
framing $(-1)$ and valency at most two, providing the three cases
above (when the valency is zero, one or two).

The first case (of a disjoint vertex with framing $(-1)$) has been
considered in Corollary~\ref{c:kulonegy}, and was shown not to change
the lattice homology.  Next we turn to the invariance under the
blow-up of a vertex. Suppose now that $G$ is a given graph with vertex
$v_0$.  Construct $G'$ by adding a new vertex $e$ with framing $(-1)$
to $G$, connect $e$ to the vertex $v_0\in \Vertices (G)$ and change
the framing of $v_0$ from $m_0$ to $m_0-1$.

\begin{thm}[\cite{lattice}]\label{thm:leaflefuj}
  The lattice homologies $\HFmComb (G)$ and $\HFmComb (G')$ of the
  graphs $G$ and $G'$ are isomorphic. Consequently lattice homology is
  invariant under blowing up a vertex.
\end{thm}
\begin{proof}
  We start by constructing auxiliary graphs for the proof.  Let the
  graph $G''$ be defined by simply adding a new vertex $e$ to $G$ with
  framing $(-1)$, without adding any new edge (or change the framing
  of $v_0$). By dropping the framing of $v_0$ from $G,G', G''$ we get
  the three graphs $\Gamma _{v_0}, \Gamma _{v_0}'$ and $\Gamma
  _{v_0}''$.  Obviously, $\Gamma _{v_0}'$ is the connected sum of
  $\Gamma _{v_0}$ with the first graph of Example~\ref{ex:trivik},
  while $\Gamma _{v_0}''$ is the connected sum of $\Gamma _{v_0}$ with
  the second graph in Example~\ref{ex:trivik}.  This fact implies then
  that $\Gamma ' _{v_0}$ can be identified with the graph $\Gamma
  _{v_0}^+$ and $\Gamma '' _{v_0}$ with $\Gamma _{v_0}^d$ of
  Corollary~\ref{c:egyenlok}. According to the corollary, therefore
  the master complexes of $\Gamma _{v_0}'$ and of $\Gamma _{v_0} ''$
  coincide.  This means that we can easily relate the surgeries on
  $v_0$ in $\Gamma _{v_0}'$ and in $\Gamma _{v_0} ''$. In fact, all
  the complexes $T_i$ and $B$ appearing in the corresponding mapping
  cones are identical, but there is a difference between the maps
  $\partial _2$ on $T_i$. To see the difference, notice that if
  $\Sigma '$ (and $\Sigma ''$) denotes the homology class of $\Gamma
  _{v_0}'$ (and $\Gamma _{v_0}''$, resp.) we fixed by
  Equation~\ref{eq:DefSigma} to define the Alexander filtration, then
  the new vertex $e$ is with multiplicity 0 in $\Sigma ''$ (since its
  multiplicity is simply $-e\cdot \Sigma ''$), while it is with
  multiplicity 1 in $\Sigma '$.  Therefore $v_0^*(\Sigma
  ')=v_0^*(\Sigma '')+1$, hence the map $\partial _2''$ for $\Gamma
  _{v_0}''$ from $T_i$ points to the same $B_j$ as the map $\partial
  _2 '$ for $\Gamma _{v_0}'$ if the framing fixed on $v_0$ in $\Gamma
  _{v_0}'$ is one less than the framing of $v_0$ in $\Gamma _{v_0}''$.
  Consequently, for framing $m_0-1$ on $v_0\in \Gamma _{v_0}'$ and
  $m_0$ on $v_0\in \Gamma _{v_0}''$ the two mapping cones coincide,
  providing an isomorphism of the corresponding homologies. Now after
  performing the surgery on $v_0$ in $\Gamma _{v_0}''$, by
  Corollary~\ref{c:kulonegy} we can simply remove the disjoint vertex
  $e$, concluding the proof of the theorem.
\end{proof}

\begin{rems}\label{rem:tobbblow}
\begin{itemize}
\item A simple adaptation of the above argument shows that if we blow
  up a vertex once, and then blow up the new edge, the lattice
  homology remains unchanged. Indeed, the argument proceeds along the
  same line, with the modification that instead of taking the
  connected sum of $\Gamma _{v_0}$ with the first graph of
  Example~\ref{ex:trivik}, we use $\Gamma ^1_{v_0}$ of
  Remark~\ref{rem:gkknot}. (The computation of the master complex of
  this graph is outlined in Remark~\ref{rem:masszam}.) When applying
  the surgery formula, we need to keep track of the homology class
  $\Sigma$ used in the definition of the Alexander filtration exactly
  as it is discussed above.

\item
  Notice that by the repeated application of the above procedure, we can
  turn the vertex $v$ into a good vertex without chaning the lattice
  homology of the graph (on the price of introducing many
  $(-1)$-framed vertices with valency two):
  each time we apply the double blow-up on $v$ we increase its valency by 
  one, while decrease its framing by two. In fact, by considering
  $\Gamma ^k _{v_0}$ of Remark~\ref{rem:gkknot} for $k\geq 0$, the
  same argument shows that the repeated blow-up of the edge connecting
  $v$ and the $(-1)$-framed new vertex does not change the lattice
  homology. Nevertheless, the value of the framing $m_v$ of $v$ drops
  by $k$ while the valency $d_v$ increases by 1, hence for $k$ large
  enough the vertex $v$ will become a good vertex (while the $(-1)$-framed vertex
  next to it will be a bad vertex). We will apply this trick in
  our forthcoming arguments.
\end{itemize}
\end{rems}

The verification of the fact that the blow-up of an edges does not
change lattice homolgy requires a much longer preparation.  The idea
of the proof is that we consider one end of the edge we are about to
blow up, drop its framing and try to compare its filtered lattice
chain complex before and after the blow-up. The graph (with this
distinguished vertex) is the connected sum of two of its subgraphs,
one of which is not affected by the blow-up, while the other changes
by blowing up the edge connecting the distiguished vertex to the rest
of the graph. In order to show that the master complexes of the graphs
before and after this blow-up are filtered chain homotopic, we will
reprove Theorem~\ref{thm:leaflefuj} by describing an explicit chain
homotopy equivalence of the background lattice homologies, which
(after suitable adjustments) will indeed respect the Alexander
filtrations.

We start with the definition of a contraction map in a general
situation, and we will turn to the description of the chain
homotopy equivalences and the filtrations after that. Consider therefore a
plumbing graph $G$ and a vertex $v$.  Let the framing $v^2=m_v$ be
denoted by $-k$.  In the next theorem we will assume that the vertex
$v$ is \emph{good}, that is, its framing $m_v$ and its valency $d_v$
satisfy $d_v+m_v\leq 0$. We will use this condition through the
following result:
\begin{lem}\label{lem:ehatok}
Suppose that $v$ is a good vertex of $G$ with $v^2=-k$. Then for any generator
$[K,E]$ with $v\in E$ we have that 
\begin{itemize}
\item $a_v[K,E]=0$ once $K(v)\geq -v^2=k$ and
\item $b_v[K,E]=0$ once $K(v)\leq v^2=-k$.
\end{itemize}
\end{lem}
\begin{proof}
  Recall that $A_v[K,E]=\min \{ f(K,I)\mid I \subset E-v\}$ while
  $B_v[K,E]= \min \{ f(K,J)\mid v\in J \subset E\}=\min \{
  f(K,I)+\frac{1}{2}(v^2+K(v)+2 {\rm {deg}}_Iv)\mid I\subset E-v\}$,
  where ${\rm {deg}}_Iv$ denotes the number of vertices in $I$
  connected to $v$. Since ${\rm {deg}}_Iv\geq 0$ for all $I$, if
  $v^2+K(v)\geq 0$ then $A_v\leq B_v$, and hence $a_v[K,E]=0$.  If
  $K(v)\leq v^2$, then $v^2+K(v)+2{\rm {deg}}_I v\leq 2v^2+2{\rm
    {deg}}_Ev= 2(m_v+d_v)$. Since $v$ is a good vertex, this
  expression is nonpositive, implying that $A_v\geq B_v$, which then
  means that $b_v[K,E]=0$.
\end{proof}
\begin{rem}
  For the classes $[K,E]$ with $K(v)\in (-k, k)$ the question of which
  of $a_v$ and $b_v$ is zero, is much more complicated.  For example,
  if $G$ is a tree on 3 vertices $\{ v, v_1, v_2\}$, with the two
  leaves $\{ v_1, v_2\}$ of framing $(-2)$ and the third vertex $v$ of
  framing $(-4)$ and $K$ is $0$ on the leaves and 2 on $v$, then
  $b_v[K, v]=0$ while $b_v[K, \{ v, v_1, v_2\} ]=1$.
\end{rem}

Suppose now that $[K,E]$ is given, and assume  that the good
vertex $v$ is not in $E$.  The above lemma implies that for the
unique value $i_0=i_{K,v}$ with the property that $(K+2i_0v^*)(v)\in
[-k, k)$, we have that $a_v[K+2iv^*, E\cup v]=0$ once $i>i_0$ and
  $b_v[K+2iv^*, E\cup v]=0$ once $i<i_0$. Also, one of $a_v[K+2i_0v^*,
    E\cup v]$ and $b_v[K+2i_0v^*, E\cup v]$ is equal to zero.
\begin{defn} The generator $[K,E]$ is of type-$a$ if $a_v[K+2i_0v^*,
  E\cup v]=0$ and of type-$b$ if $a_v[K+2i_0v^*, E\cup v]> 0$. Let
  $T=T_{[K,E]}$ be equal to 1 if $[K,E]$ is of type-$b$ and $-1$ if it
  is of type-$a$.
\end{defn} 
Consider the map $H_0\colon \CFmComb (G)\to \CFmComb (G)$ defined as
\[
H_0[K,E]=
\left\{
\begin{array}{ll}
  0 & {\text{if $v\in E$ or $(T-2)k\leq K(v)<Tk$}} \\
  \\
{[K,E\cup v]} 
  & {\text{if $v\not\in E$ and $K(v)\geq Tk$}} \\
   \\
 { [K-2v^*,E\cup v]} 
  & {\text{if $v\not\in E$ and $K(v)<(T-2)k$}} \\
\end{array}
\right.
\]

\begin{lem}\label{lem:mmas}
The Maslov grading of 
$H_0[K,E]$ (if this term is not zero) is equal
to $\grad ([K,E])+1$.
\end{lem}
\begin{proof}
By considering $\partial _v H_0[K,E]$ (where $\partial _v$ denotes the components of $\partial$ when we delete $v$ from the set), we see that the component with vanishing $U$-power is
exactly $[K,E]$, hence the claim follows from the fact that $\partial$ drops Maslov grading
by one.
\end{proof} 

\begin{defn}
  Define the map
   $C_0\colon \CFmComb (G)\to
  \CFmComb (G)$ by
\[
C_0[(K,p),E]=[(K,p),E]+\partial \circ H_0 [(K,p),E]+ H_0\circ \partial 
[(K,p),E] .
\]

Notice that for each $[K,E]$ there is $N=N_{[K,E]}$ with the property
that the $N^{th}$ iterate of $C_0$ stabilizes; i.e.
writing
$C_0^n[K,E]=\overbrace{n}{C_0\circ \dots \circ C_0}[K,E]$, we have that
$C_0^N[K,E]=C_0^{N+1}[K,E]$. Thus, it makes sense to talk about the infinite
iterate
$C_0^{\infty}$. We call this stabilized map $C_v=C^{\infty}_0$ the {\em contraction map}.
\end{defn}

Notice  that both $C_0$ and $C_v$ preserve the
Maslov grading (in the sense that if $C_v[K,E]\neq 0$ then its Maslov
grading is equal to the Maslov grading of $[K,E]$).

\begin{thm}\label{thm:contraction}
For a good vertex $v$ the contraction map $C=C_v$ satisfies
$ C[K,E]=0$ if $v\in E$.
\end{thm}
\begin{proof}
Consider the map $H=H_v\colon \CFmComb (G)\to \CFmComb (G)$ defined as
\[
H[K,E]=
\left\{
\begin{array}{ll}
  0 & {\text{if $v\in E$ or $(T-2)k\leq K(v)<Tk$}} \\
  \\
  \sum_{i=0}^{t-1} U^{s_i} \otimes [K+2iv^*,E\cup v] 
  & {\text{if $v\not\in E$ and $K(v)=q+2tk$}} \\
  & {\text{with $q\in [(T-2)k,Tk), t>0$}} \\
  \\
  \sum_{i=0}^{-t-1} U^{r_i} \otimes [K-2(i+1)v^*,E\cup v] 
  & {\text{if $v\not\in E$ and $K(v)=q+2tk$}} \\
  & {\text{with $q\in [(T-2)k,Tk), t<0$}}
\end{array}
\right.
\]
where $s_0=1, s_{i+1}=s_i+b_v[K+2iv^*, E\cup v ]$ and 
$r_0=0, r_{i+1}=r_i+a_v[K-2(i+1)v^*, E\cup v]$.

It is easy to see that $C[K,E]=[K,E]+\partial \circ H [K,E]+
H\circ \partial [K,E]$. If $v\in E$, then the middle term of
this expression is obviously zero. Suppose
first that $K(v)=q+2tk$ (with $q\in [(T-2)k, Tk)$) and $t$ is positive. Then
we need to consider only those parts of $\partial [K,E]$ where the set
$E-w$ does not contain $v$ (since for $v\in E-w$ the map $H$ will annihilate 
the term anyhow), implying that
\begin{equation}\label{eq:cek}
H(\partial [K,E])=\sum _{i=0}^{t-1}U^{s_i}[K+2iv^*, E]+U^{b_v[K,E]}
\sum _{i=0}^{t-2}U^{s_i'}[K+2v^*+2iv^*, E] .
\end{equation}
(Notice that $a_v[K,E]=0$ in this case, and also the second summation
goes for one less term, since $t$ for $K+2v^*$ is one less than for
$K$.)  It is clear that terms come in pairs and since they have equal
Maslov gradings, the $U$-powers necessarily match up. (The actual
identities here can be checked by direct and sometimes lengthy
computations; since the principle based on Maslov gradings is much
shorter, we will not provide those explicite formulae here.) The term
corresponding to $i=0$ in the first sum has no counterpart, hence the
sum of~\eqref{eq:cek} reduces to $[K,E]$, therefore $C[K,E]=0$ follows
at once.  The exact same computation for $K(v)=q+2tk$ with $t\leq 0$
(after similar cancellations) provides $C[K,E]=0$ in this case as
well.
\end{proof}

\begin{exa}
  We consider the following special case: suppose that $v=e$ is a leaf
  of the graph with $e^2=-1$. Since this vertex is good, the previous
  results apply.  The value of $C_e$ can be determined provided we
  compute the types of all the elements appearing in this
  computation. It is hard to give a closed formula, therefore we
  will just outline the computation and highlight the important
  features of the resulting expressions. Recall that
  $C_e=Id+\partial \circ H_e+H_e\circ \partial$.  Suppose that
  $(K,p,j)$ is a characteristic cohomology class, where $j$ is the
  value on $e$, $p$ is the value on the unique vertex $v$ connected to
  $e$ and $K$ is the restriction of the class to $G-v-e$.  In
  computing the value $C_e[(K,p,j),E]$, we start with determining the
  boundary of $H_e[(K,p,j,),E]$. The terms in $\partial (H_e [(K,p,j),
  E])$ are of two types: for two terms the set will be equal to $E$
  (when we take $\partial _e$) while for all the others the set will
  be of the shape $E-w\cup e$ for some $w\in E$.  The first type of
  contribution equals either $[(K,p,j), E]+U^x[(K,p+j+1,-1),E]$ or
  $[(K,p,j), E]+U^y[(K,p+j+3,-3),E]$ (depending on whether
  $[(K,p,j),E]$ is of type-$b$ or of type-$a$). Here the $U$-powers
  are determined by the requirement that the Maslov gradings of the
  terms are equal to the Maslov grading of $[(K,p,j),E]$ (and we do
  not describe their actual values here explicitly).

  The further terms involve sets of the form $E-w\cup e$.  We need to
  distinguish two cases, depending on whether $e$ and $w$ are
  connected or not. Suppose first that $w$ is not connected to
  $e$. Then each such term appears once in $\partial _w\circ H$ and
  once in $H\circ \partial _w$, and the terms cancel if the type of
  the element is the same as the type of $[(K,p,j),E]$ and do not
  cancel otherwise. The case when $w$ is connected to $e$ is slightly
  different, since in computing $H\circ \partial _w$ a further term
  appears (since in one component of $\partial _w$ the value of the
  cohomology class on $e$ becomes higher). These terms will be
  analyzed in detail in the proof of Proposition~\ref{prop:respr}.

  In particular, since the type of $[(K,p,j)+2w^*,E]$ is the same as
  the type of $[(K,p+j+1,-1)+2w^*,E]=[(K,p,j)+2w^*+2ne^*, E]$ (with
  $j=2n-1$), it follows that
\begin{equation}\label{eq:ugyanaz}
C_e [(K,p,j),E]=C_e [(K,p+j+1,-1), E].
\end{equation}
 \end{exa}
After these preparations we return to relating the lattice homology of
a graph and its blow-up. We will reexamine the blow-up of a vertex ---
the filtered version of the resulting identity will be used in the
proof of the invariance under the blow-up of an edge.

Suppose that $G$ is a given framed graph containing the vertex $v$,
and $G'$ is given by blowing up $v$. As before, the new vertex
introduced by the blow-up will be denoted by $e$. Recall that the
framing of $v$ in $G'$ is one less than its framing in $G$.  In the
following we write characteristic vectors for $G'$ as triples
$(K,p,j)$, where $K$ denotes the restriction of the characteristic
vector to the subspace spanned by the subgraph $G-v=G'-\{ e,v\}
\subset G'$, $p$ denotes the value of the characteristic vector on the
distinguished vertex $v$, and $j$ denotes the value on the new vertex
$e$. Similarly, characteristic vectors on $G$ will be denoted by
$(K,p)$, where $p$ is the value on $v$ and $K$ is the restriction to
$G-v$.

We define the ``blow-down'' map $P\colon \CFmComb(G')\longrightarrow
\CFmComb(G)$ by the formula
\[
P[(K,p,j),E]= \left\{\begin{array}{ll}
U^{s}\otimes [(K,p+j),E] & {\text{if $e\not\in E$}} \\
0	& {\text{if $e\in E$,}} 
\end{array}\right.
\]
where $s=g[(K,p+j),E]-g[(K,p,j),E]+\frac{j^2-1}{8}.$ The value of $s$
is taken to ensure that the Maslov grading of $P[(K,p,j),E])$ is equal
to the Maslov grading of $[(K,p,j),E]$. Since for any subset $E$ not
containing $e$ the inequality $f([(K,p+j),I])\geq
f([(K,p,j),I])-\frac{j^2-1}{8}$ holds for $I\subset E$, it follows
that $s\geq 0$.

\begin{lem}
\label{lemma:PisChain}
The blow-down map $P$ is a chain map.
\end{lem}
\begin{proof}
We wish to prove 
\begin{equation}
\label{eq:ChainMap}
\partial\circ P[(K,p,j),E] = P\circ \partial'[(K,p,j),E].
\end{equation}
First, we consider the case where $e\in E$.  In this case the
left hand side is zero, while
\begin{eqnarray*}
P\circ \partial'[(K,p,j),E] &=&
P(U^{a_e[(K,p,j),E]}\otimes [(K,p,j),E-e]) \\
&& +P (U^{b_e[(K,p,j),E]}\otimes [(K,p+2,j-2),E-e]) \\
&=& U^{d_1}\otimes [(K,p+j),E-e] + U^{d_2}  [(K,p+j),E-e],
\end{eqnarray*}
for some appropriately chosen $d_1$ and $d_2$.
By the equality of Maslov gradings the two expressions are equal, and
hence the terms obviously cancel.

Next, suppose that $e\not\in E$. 
Observe that
\begin{eqnarray*}
\lefteqn{P\circ \partial'[(K,p,j),E]=} \\
&& 
\sum_{w\in E} U^{c_1(w)}\otimes [(K,p+j),E-w] 
+ U^{d_1(w)} \otimes[(K,p+j)+2w^*,E-w],
\end{eqnarray*}
and 
\begin{eqnarray*}
\lefteqn{\partial \circ P[(K,p,j),E]=} \\
&& \sum_{w\in E} U^{c_2(w)}\otimes [(K,p+j),E-w] 
 + U^{d_2(w)}\otimes [(K,p+j)+2w^*,E-w],
\end{eqnarray*}
Once again, the argument based on Maslov gradings shows that $c_1(w)=
c_2(w)$ and $d_1(w)=d_2(w)$, completing the verification of
Equation~\eqref{eq:ChainMap}, hence concluding the proof of the lemma.
\end{proof}

Define the ``blow-up'' map $R\colon \CFmComb(G) \longrightarrow
\CFmComb(G')$ by the formula
\[
R([(K,p),E])=C_e([K,p+1,-1),E]).
\]
Since $e\not\in E$, we have that $g[(K,p+1,-1),E]=g[(K,p),E]$, implying 
that the Maslov grading of $R[(K,p),E]$ (if this term is not zero) is equal 
to the Maslov grading of $[K,E]$.
\begin{lem}
The map $R$ is a chain map.
\end{lem}
\begin{proof}
Let us first consider the map $Q\colon \CFmComb (G) \to \CFmComb (G')$
given by the formula
\[
Q([(K,p),E])=
\left\{\begin{array}{ll}
{[(K,p+1,-1),E]} & {\text{if $v\not\in E$}} \\
{[(K,p+1,-1),E]} + \\
+ U^r\otimes [(K,p+1,-1)+2v^*,(E-v)\cup e]
& {\text{if $v\in E$,}} 
\end{array}\right.
\]
where $r=b_v[(K,p+1,-1), E\cup e]\geq 0$.  The map $Q$ preserves the
Maslov gradings: For the first term we appeal to the observation
that when $e\not \in E$ then $g[(K,p+1,-1),E]=g[(K,p),E]$. For the
second term the exponent $r$ can be shown to be equal to
$B_v[(K,p+1,-1),E\cup e]-g[(K,p),E]$, since $B_v[(K,p+1,-1), E\cup
e]=\min \{f((K,p+1,-1), I)\mid v\in I\subset E\cup e\}$ and
$g[(K,p+1,-1),E]=g[(K,p),E]$. Now the difference of the Maslov
gradings of $[(K,p),E]$ and of $[(K,p+1,-1)+2v^*, (E-v)\cup E]$ can be
easily identified with twice the above difference, concluding the
argument.

Notice that $R=C_e\circ Q$ (since $C_e$ maps the term with set
containing $e$ to zero). Since $C_e$ is a chain map, we only need to
verify that $Q$ is a chain map. As in \eqref{eq:ChainMap}, we need to
verify that 
\begin{equation}\label{eq:kiem}
Q\circ \partial [(K,p), E]=\partial ' \circ Q[(K,p),
E]. 
\end{equation}
Consider first the components of the boundary with set equal to $E-w$
for some $w\in E$ distinct from $v$. On both sides these elements are
of the form $[(K,p), E-w]$ and $[(K,p)+2w^*, E-w]$ (multiplied with
some $U$-powers).  Since the terms coincide, and the Maslov gradings
are equal, the $U$-powers should be equal as well, verifying the
equation for such terms.  The above argument verifies the required
identity of~\eqref{eq:kiem} in the case $v\not \in E$.

Assume now that $v\in E$ and consider $Q(\partial _v[(K,p),E] )$. We
claim that it is equal to
\begin{equation}\label{eq:hatarszam}
(\partial _v+\partial _e)([(K,p+1,-1),E]+ U^r\otimes
[(K,p+1,-1)+2v^*,(E-v)\cup e]).
\end{equation}
Indeed, $\partial _v ([(K,p), E]=U^{a_v[(K,p),E]}[(K,p),E-v]+
U^{b_v[(K,p),E]}[(K,p)+2v^*,E-v]$, and its $Q$-image is simply
\[
 U^{a_v[(K,p),E]}[(K,p+1,-1),E-v]+U^{b_v[(K,p),E]}[(K+2v^*,p+2(v^2)+1, -1),E-v].
 \]
Now writing out \eqref{eq:hatarszam} we get four terms:
\[
U^{a_1}[(K,p+1,-1), E-v]+U^{b_1}[(K, p+1, -1)+2v^*, E-v]+
\]
\[
+ U^{a_2}[(K,p+1, -1)+2v^*, E-v]+U^{b_2}[(K,p+1,-1)+2v^*+2e^*, E-v].
\]
(As usual, we did not specify the actual $U$-powers, which are 
dictated by the fact that the maps preserve the Maslov gradings.)
The second and the third term cancel each other, while the first and
the fourth are equal to the terms appearing in $Q(\partial _v [(K,p),E])$.
(In comparing the fourth term above to the second term in
$Q(\partial _v [(K,p),E])$ one needs to take the change of $v^2$ into account.)
This last observation then 
concludes the proof of the lemma.
\end{proof}
\begin{thm}
\label{thm:CombBlowUp}
$P$ and $R$ are chain homotopy equivalences.
\end{thm}
\begin{proof}
  First we examine the composition $P\circ R$. We claim that since
  $g[(K,p+1,-1),E]=g[(K,p),E]$, we have that $P\circ R=\Id$. Indeed,
  applying $P$ to $C_e[(K,p+1,-1),E]$, all terms with set containing
  $e$ will be mapped to zero, while the remaining single term is
  either $P[(K,p+1,-1),E]$ or $P[(K,p+3,-3),E]$ (with some $U$-power
  in front).  In both cases the image is $[(K,p),E]$ (multiplied with
  some power of $U$).  Since the maps preserve the Maslov grading,
  the power of $U$ is equal to zero, hence $P\circ R$ is equal to the
  identity.  
  
  Regarding the composition $R\circ P$, we claim that $R\circ P=C_e$.
  If $e\in E$ then both $P$ and $C_e$ vanish, hence the equality
  holds. The identity then simply follows from the observation of
  \eqref{eq:ugyanaz} that $C_e[(K,p,j),E]=C_e[(K,p+j+1,-1),E]$ and
  from the fact that $P,R$ and $C_e$ all preserve Maslov gradings.
  Now $H=H_e$ furnishes the required chain homotopy between $R\circ P$
  and the identity.
  \end{proof}
  Notice that the chain homotopies found in
  Theorem~\ref{thm:CombBlowUp}  provide a further proof of
  Theorem~\ref{thm:leaflefuj}. Now, however, we would like to consider
  two new graphs (with unframed vertices in them): let $\Gamma _{v_0}$
  be the graph we get from $G$ by attaching a new vertex $v_0$ and a
  new edge connecting $v$ and $v_0$ to it.  Similarly, $\Gamma
  _{v_0'}'$ is constructed from $G'$ by adding a new vertex $v_0'$ and
  an edge connecting $v_0'$ and $e$.  (Alternatively, $\Gamma
  '_{v_0'}$ can be given by blowing up the egde of $\Gamma _{v_0}$
  connecting $v$ and $v_0$.)  Our next goal is to prove

\begin{thm}\label{thm:elfelf}
  Suppose that $v$ is a good vertex of $G$, that is, for the framing
  $m_v$ and valency $d_v$ of $v$ we have $m_v+d_v\leq 0$.  Then the
  filtered lattice chain complex $(\CFmComb (G), A_{v_0})$ of $\Gamma
  _{v_0}$ is filtered chain homotopic to the filtered lattice chain
  complex $(\CFmComb (G'), A_{v_0'})$ of $\Gamma _{v_0'}'$.
\end{thm}

Before turning to the proof of this result, we show that (compositions of) the maps
introduced earlier are, in fact, filtered maps.
\begin{prop}
Suppose that $v$ is a good vertex of a plumbing graph $G=\Gamma _{v_0}-v_0$ and $v_0$ 
is connected only to $v$. Then
the map $C_v$ is a  filtered chain map, which is filtered chain 
homotopic to the identity.
\end{prop}
\begin{proof}
We will show that the map $C_0$ is a filtered chain map, chain homotopic
to the identity --- 
obviously by iteration both statements of the proposition follow from this result.
In turn, to show the statement for $C_0$, we only need to 
show that the homotopy $H_0$ respects the Alexander filtration.
We claim that the Alexander gradings of $[K,E]$ (with $v\not\in E$) and 
$H_0[K,E]$ (when this latter term is nonzero) are equal. 

Assume  therefore that $v\not \in E$ and $K(v)\geq k$.
In this case by Lemma~\ref{lem:ehatok} both $a_v[K,E\cup v]$ and 
$a_v[K+2v_0^*,E\cup v]$ are zero. The first fact is used in the definition of the
contraction, while the second one shows (by Lemma~\ref{lem:filtralt}) that the difference
of the  Alexander gradings of $[K,E]$ and of $[K,E\cup v]$ 
 is zero, concluding the argument in this case. 
Similarly, if $[K,E]$ satisfies $K(v)<-k$ then $(K+2v_0^*)(v)\leq -k$, hence again
both $b_v[K,E\cup v]$ and $b_v[K+2v^*_0, E\cup v]$ vanish, providing the same conclusion.

Assume now that $K(v)\in [-k,k)$. First we show that if $[K,E]$ is of type-$a$ 
(that is, $a_v[K,E\cup v]=0$) then the
extension $L_{[K,E\cup v]}$ (provided by Lemma~\ref{lem:letezik}) on $v_0$ vanishes.
Indeed, $a_v[K,E\cup v]=0$ means that the minimum $g[K,E\cup v]$ is attained by a set 
$I\subset E\cup v$ which does not contain $v$. For this set
$f(K,I)=f(K+2v_0^*,I)$ (since $v$ is the only vertex connected to $v_0$), therefore
$g[K,E\cup v ]=g[K+2v_0^*,E\cup v]$, implying that the extension 
$L_{[K,E\cup v]}$ is zero. This fact then shows that $A([K,E])=A([K,E\cup v])=A(H_0[K,E])$.

Suppose now that $a_v[K,E\cup v]>0$ (that is, $[K,E]$ is of type-$b$). Then 
obviously $b_v[K,E\cup v]=0$ and we wish to show that 
$A[K+2v^*, E]=A[K,E\cup v]$. First we show that 
$A[K,E\cup v]=A[K,E]-1$. Indeed, the parts of the definition of the Alexander grading 
involving $K$ and $\Sigma$ are the same for both. The extension $L_{[K,E]}(v_0)$ is obviously
zero since $v\not\in E$, hence we only need to check that $L_{[K,E\cup v]}=-2$. The assumption
$a_v[K,E\cup v]>0$ then implies that if $f(K,I)=g[K,E\cup v]$ for some
$I\subset E\cup v$ then $v\in I$. Therefore $g[K+2v_0^*,E\cup v]=g[K,E\cup v]+1$, hence the
claim follows. Now our computation will be complete once we show that 
$A([K+2v^*, E])=A([K,E])-1$ once $v\not \in E$. This equation easily follows from the fact that
the extension $L$ in both cases vanishes on $v_0$, 
while $(K+2v^*)(\Sigma -v_0)=K(\Sigma -v_0)+
2v^*(\Sigma) -2v\cdot v_0=K(\Sigma -v_0)-2$.
\end{proof}

\begin{prop}\label{prop:respcp}
The map $C\circ P$ is a filtered chain map.
\end{prop}
\begin{proof}
  Obviously both maps are chain maps, hence we only need to show that
  the composition of the two maps does not increase the Alexander
  filtartions.  If $e$ or $v$ is in $E$, then the
  composition maps $[K,E]$ to zero. Hence we only need to deal with
  those generators $[K,E]$ for which $e,v\not \in E$. We claim that
  for those elements $P$ does not increase the Alexander
  grading. (Since $C$ is a filtered map, this implies that so is
  $C\circ P$.)  Since $e,v\not \in E$, it follows that
  $P[(K,p,j),E]=U^s\otimes [(K,p+j),E]$, where (again, by $e,v\not \in
  E$) the term $s$ is equal to $\frac{j^2-1}{8}$. Since the extensions $L(v_0)$ in
  both cases are equal to 0 (with the choice $v_0^2=0$), the
  inequality
\[
A[(K,p,j),E]\geq -s+A[(K,p+j),E]
\]
is equivalent to $j+1\geq -\frac{j^2-1}{4}$, which obviously holds for
every odd integer $j$.
\end{proof}

\begin{prop}\label{prop:respr}
The map $R$ is a filtered chain map.
\end{prop}
\begin{proof}
Once again, we already showed that $R$ is a chain map, hence we only
need to verify that it respects the Alexander filtrations: we need to
compare the Alexander grading of $[(K,p),E]$ and of
$C_e[(K,p+1,-1),E]$. Before giving the details of the argument, notice
that if $\Sigma = v_0+a_v\cdot v + \sum _{j=1}^n a_j \cdot v_j$ is the
homology class used in the definition of $A$ in $G$
(cf. Equation~\ref{eq:DefSigma}), then in $G'$ the corresponding
element is $\Sigma '=v'_0+(1+a_v)\cdot e+a_v\cdot v +\sum _{j=1}^n a_j
\cdot v_j$. Consequently $(\Sigma ')^2=\Sigma ^2 +1$.

By its definition, $L_{[(K,p),E]}(v_0)$ is either 0 or $-2$.
When $L_{[(K,p),E]}(v_0)=0$ ,
the proof of the claim is rather simple: in fact, $A([(K,p+1,-1),E])\leq 
A[(K,p),E])$, since the values  $(K,p)(\Sigma -v_0)+\Sigma ^2$ and
$(K,p,j)(\Sigma' -v_0)+(\Sigma ')^2$ coincide, we add 0 to the first and  the nonpositive term
$L_{[(K,p,j),E]}(v_0')$ to the second term. 

Suppose now that $L_{[(K,p),E]}(v_0)=-2$. By its definition this means
that
\[
g[(K,p),E]<g[(K,p+2),E], 
\]
therefore the set $I$ with $f(K,I)=g[(K,p),E]$ contains $v$ (and, in
particular, $E$ should contain $v$). We know that $g[(K,p),E]=g[(K,
  p+1,-1), E]$, but since $I$ (on which the minimum is taken) contains
$v$, we further deduce that $g[(K, p+1,-1), E]=g[(K,p+1,-1),E\cup e]$,
or equivalently $a_e[(K,p+1,-1),E\cup e]=0$.  This then means that
$[(K,p+1,-1),E]$ is of type-$a$, hence the image $C_e[(K,p+1,-1), E]$
has $[(K,p+3,-3),E]$ as the component having $E$ as the set.  Since
$A[(K,p+3,-3),E]=A[(K,p+1,-1),E]-1=A[(K,p),E]$ (this last equality
holding because $L_{[(K,p),E]}(v_0)=-2$), we see that this component
of $R[(K,p),E]$ has Alexander grading at most the Alexander grading of
$[(K,p),E]$.

We still need to examine the further components of
$C_e[(K,p+1,-1),E]$. These terms are all of the form $[L,E-w\cup e]$
for some $w\in E$. Assume first that $w\neq v$, that is, $w$ and $e$
is not connected by an edge. These terms come from the parts of $C_e$
given by
\[
\partial _w [(K,p+1,-1), E\cup e] \quad {\mbox {and}}\quad
H\circ\partial _w [(K,p+1,-1),E] .
\]
The contributions of these terms depend on the fact whether
$[(K,p+1,-1),E-w]$ (and similarly, $[(K,p+1,-1)+2w^*, E-w]$) is of
type-$a$ or of type-$b$. If the term is of type-$a$, then the
contributions cancel. On the other hand, if the element
$[(K,p+1,-1),E-w]$ is of type-$b$, then we will have a contribution of
the form $U^{a_{w}}[(K,p+1,-1),E-w\cup e]$ in $C_e$. Since it is of
type-$b$, when computing $g([(K,p+1,-1), E-w\cup e]$ the minimum is
taken on a set $I\subset E-w\cup e$ containing $e$. This means
that the value of the extension $L$ on $v_0'$ is $-2$ in this case, ultimately showing
that the Alexander grading of this term is not greater than
$A[(K,p),E]$.  The same argument works for the terms of the  shape 
$U^b[(K,p+1,-1)+2w^*, E-w\cup e]$.

When $w=v$, a further term appears. The argument of showing the
decrease of the Alexander grading proceeds roughly as it is explained
above.  In particular, $\partial _v[(K,p+1,-1), E\cup e]= U^a
[(K,p+1,-1),E\cup e -v]+U^b [(K,p+1,-1)+2v^*, E\cup e-v]$, while
similar terms appear as $H(U^{a'}[(K,p+1,-1), E-v])$ and
$H(U^{b'}[(K,p+1,-1)+2v^*, E-v])$.  The actual values of these two
terms depend on the types of the generators.  If $[(K,p+1,-1),E-v]$ is
of type-$a$, then its $H$-image cancels $U^a[(K,p+1-1), E\cup e-v]$.
If, however, $[(K,p+1,-1),E-v]$ is of type-$b$, that is,
$a_e[(K,p+1,-1),E\cup e-v]>0$, then $U^a[(K,p+1-1), E\cup e-v]$
survives. In this case the extension of the cohomology class to $v_0$
is $-2$, since the property $a_e[(K,p+1,-1),E\cup e-v]>0$ shows that
the minimum giving $g[(K,p+1,-1), E\cup e-v]$ is taken on a set which
contains $e$.  We also need to address the possible cancellation of
the terms involving the cohomology classes of the form
$(K,p+1,-1)+2v^*$. If $[(K,p+1,-1)+2v^*, E-v]$ is of type-$b$, then
these  terms cancel. On the other hand, if this generator is
of type-$a$, then we see a new phenomenon, since in this case its
$H$-image involves two terms, one of them being cancelled by the
relevant part of $\partial _v [(K,p+1,-1), E\cup e]$, but the other
one must be delt with.

Hence we need to examine the term $U^x[(K,p+1, -1)+2v^*+2e^*, E\cup
e-v]$ where $x=b_v[(K,p+1,-1), E]+b_e[(K, p+1, -1)+2v^*, E\cup e]$ in
the case when $a_e[(K,p+1,-1)+2v^*+2e^*, E\cup e-v]=0$. In this case,
however, the Alexander grading is clearly not more than $A([(K,p),E]$:
by evaluating $(K,p+1, -1)+2v^*+2e*$ on $\Sigma ' -v_0$ we get
$(K,p)(\Sigma -v_0)-2$ (from $2e^*(-v_0)$), hence the $-2$ appearing
in the Alexander grading of $[(K,p), E]$ is compensated by this $-2$,
eventually showing that $R$ does not increase the Alexander grading.
This last observation concludes the proof.
\end{proof}

\begin{proof}[Proof of Theorem~\ref{thm:elfelf}]
The maps $R$ and $C\circ P$ provide the homotopy equivalences.
Since by Propositions~\ref{prop:respcp} and \ref{prop:respr}  
these maps respect the Alexander filtrations,
we only need to show that the two compositions are filtered chain
homotopic to the respective identities. Since $P\circ R$ is the identity map
and $C$ is filtered chain homotopic to the identity,
it follows that $C\circ P\circ R =C$  has this property.

The filtered chain homotopy between $R\circ C \circ P$ and $\Id$
can be constructed as follows:
\[
R\circ C \circ P=R\circ (Id +\partial \circ H +H\circ \partial ) \circ P=
R\circ P+ \partial \circ (RHP)+(RHP)\circ \partial .
\]
Since $R\circ P$ is equal to $C_e$, and $C_e$ is filtered chain homotopic to the identity,
we only need to check that the composition $RHP=R\circ H\circ P$ is a filtered map.
If $v$ or $e$ is in $E$, then the image of this triple composition on $[K,E]$ is 
zero. Otherwise $P$ is a filtered map  on the elements $[K,E]$ with $v\not \in E$
(as it was shown in the proof of Proposition~\ref{prop:respcp}) . 
The map $H$ also respects the Alexander filtration,
and so does $R$ (as we showed in Proposition~\ref{prop:respr}), concluding the proof.
\end{proof}

Now we are in the position of proving that lattice homology of a
negative definite tree $G$ remains unchanged when we blow up an edge
$D$ of $G$.   Let $G'$ denote the graph we get by blowing up the
edge $D$.

\begin{thm}\label{thm:veglefujj}
The lattice homologies $\HFmComb (G)$ and $\HFmComb (G')$ are
isomorphic.
\end{thm}
\begin{proof} 
Suppose that after deleting the edge $D$ (connecting the
vertices $v, w_0\in \Vertices (G)$), the graph $G$ falls into two
components $G^1$ and $G^2$, where $G^1$ contains $v$ while $G^2$
contains $w_0$.  
  Suppose first that the edge $D$ connects two vertices such that at
  least one of them is good, so we can assume that $v$ is a good
  vertex.  Define $\Gamma _{w_0}$ by deleting the framing of
  $w_0$ in $G^2$.  The graph $\Gamma _{v_0}$ is defined by attaching a new
  vertex $v_0$ to $G^1$, together with an edge connecting $v$ and
  $v_0$. (The framing of $v$ in this graph is the same as its framing
  in $G$.)  Finally, we construct $\Gamma _{v_0'}'$ from $G^1$ as follows:
  we add two vertices ($e$ and $v_0'$) to it, together with an edge
  connecting $v$ and $e$ and a further edge connecting $e$ and
  $v_0'$. At the same time, we change the framing of $v$ to one less
  than it had in $G$, and attach the framing $(-1)$ to $e$. See
  Figure~\ref{fig:rajzban}.
\begin{figure}[ht]
\begin{center}
\epsfig{file=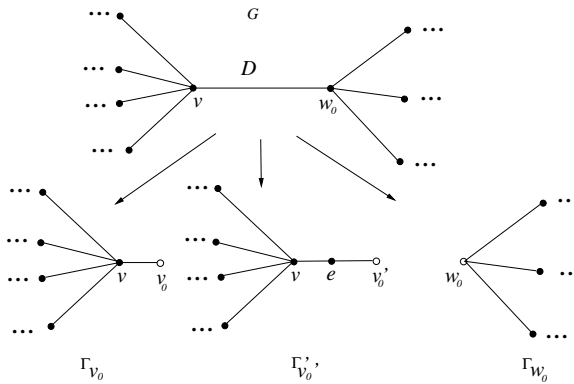, height=5cm}
\end{center}
\caption{{\bf Starting with the graph $G$ and the edge $D$ connecting
    vertices $v$ and $w_0$ we construct three further graphs.}  The
  framing of $v$ in $\Gamma _{v_0}$ is the same as in $G$ (while the
  vertex $v_0$ admits no framing). The framing of $v$ in $\Gamma _{v_0'}'$
  is one less than its framing in $G$, and $e$ is decorated by
  $(-1)$. The hollow circles refer to vertices which do not admit
  framings, hence symbolize knots in the background 3-manifolds.}
\label{fig:rajzban}
\end{figure}
It is easy to see that the connected sum of $\Gamma _{w_0}$ and
$\Gamma _{v_0}$, together with the appropriate framing on $v_0=w_0$
restores the graph $G$, while if we take the connected sum of $\Gamma
_{w_0}$ and $\Gamma _{v_0'}'$ (and attach framing one less to $v_0'$
than above) then we get the graph $G'$ which is constructed from $G$
by blowing up the edge $D$.

Let $\Delta _{w_0}$ denote the graph we get by taking the connected
sum of $\Gamma _{v_0}$ and $\Gamma _{w_0}$. Similarly, $\Delta
'_{w_0}$ is defined as the connected sum of $\Gamma _{v_0'}'$ and $\Gamma
_{w_0}$.  By Theorem~\ref{thm:elfelf} the filtered lattice chain
complexes of $\Gamma _{v_0}$ and of $\Gamma _{v_0'}'$ are filtered chain
homotopic, hence by the connected sum formula the filtered lattice
chain complexes of $\Delta _{w_0}$ and $\Delta '_{w_0}$ are filtered
chain homotopic.  As in the proof of Theorem~\ref{thm:leaflefuj}, we
notice that the homology classes $\Sigma$ and $\Sigma '$ (used in the
definitions of the Alexander filtrations for $\Delta _{w_0}$ and
$\Delta '_{w_0}$) are slightly different, with the property that
$w_0^*(\Sigma )=w_0^*(\Sigma ')+1$. As in the proof of
Theorem~\ref{thm:leaflefuj}, this identity and the homotopy
equivalence of the filtered lattice chain complexes of $\Delta _{w_0}$
and $\Delta '_{w_0}$  imply that the lattice homology of a graph
we get by attaching any (negative enough) framing to $w_0$ in $\Delta
_{w_0}$ is isomorphic to the lattice homology of the graph we get from
$\Delta '_{w_0}$ by attaching to $w_0$ a framing one less. This
exactly verifies the statement of the theorem under the hypothesis that
$v$ is a good vertex.

The general case can be reduced to the above situation by first
blowing up the vertex $v$, and then blowing up the new egde, and
repeating this two-step procedure until $v$ will be a good vertex of
the resulting graph.  According to Remark~\ref{rem:tobbblow} the
procedure does not change the lattice homology, while the above
argument shows that it remains unchanged when we blow up the edge
$D$. Finally the application of Remark~\ref{rem:tobbblow} again shows
that we can blow back down  the blow-ups we used to turn $v$ into a good
vertex.  This final step then concludes the proof of the theorem.
\end{proof}

\begin{proof}[Proof of Theorem~\ref{thm:indep}]
  Suppose that $G_1$ and $G_2$ are two negative definite plumbing
  graphs with the property that the associated 3-manifolds $Y_{G_1}$
  and $Y_{G_2}$ are diffeomorphic. By \cite{neumann} this implies that
  $G_2$ can be given from $G_1$ by a sequence of blowing up vertices,
  edges and adding disjoint $(-1)$-framed vertices, and the inverses
  of these operations.  Corollary~\ref{c:kulonegy} and
  Theorems~\ref{thm:leaflefuj} and \ref{thm:veglefujj} show that
  lattice homology is invariant under these moves, implying the
  existence of the claimed isomorphism between $\HFmComb (G_1)$ and
  $\HFmComb (G_2)$.
\end{proof}

\section{Appendix: finitely generated subcomplexes}
\label{app:contract}
In this section we show that the lattice homology $\HFmComb (G, \s )$
of a negative definite graph $G$ (with a fixed spin$^c$ structure
$\s$) is finitely generated as an $\Field [U]$-module.  (Once again,
this result was already verified by N\'emethi in \cite{lattice}.)
This finiteness result will easily follow from the corresponding
result for $\HFaComb (G, \s)$:

\begin{thm}\label{thm:fingenhat}
  Suppose that $G$ is a negative definite plumbing graph and $\s$ is a
  given spin$^c$ structure on $Y_G$. Then the homology group $\HFaComb
  (G, \s )$ is a finite dimensional $\Field$-vector space.
\end{thm}

\begin{cor}\label{cor:fingenmin}
  The lattice homology $\HFmComb (G, \s )$ of a negative definite
  plumbing graph $G$ (equipped with a spin$^c$ structure $\s$) is a
  finitely generated $\Field [U]$-module.
\end{cor}
Notice that the above corollary completes the proof of
Theorem~\ref{thm:structure}.
\begin{proof}[Proof of Corollary~\ref{cor:fingenmin}]
 Recall that there is a short exact sequence 
 \[
 0\to \CFmComb _{q+2} (G, \s ) \stackrel{\cdot U}{\to}\CFmComb _q(G, \s)\to 
 \CFaComb _q(G, \s )\to 0,
 \]
 where the first map is multiplication by $U$ and which induces a long
 exact sequence on the corresponding homologies.  (The lower indices
 now indicate Maslov gradings.) Since by Theorem~\ref{thm:fingenhat}
 the group $\HFaComb (G, \s )$ is finitely generated over $\Field$,
 for $q$ sufficiently negative we have that $\HFaComb _p (G, \s) =0$
 for all $p\leq q$. By exactness this implies that multiplication by
 $U$ provides an isomorphism once $q$ is sufficiently
 negative. Phrased differently, for $q$ negative enough, each element
 $y$ of $\HFmComb _q (G, \s)$ can be written as $y=Ux$ for some $x\in
 \HFmComb _{q+2}(G, \s )$. Since $G$ is negative definite, it is not
 hard to see from the definition of the Maslov grading that there is a
 constant $k_G$ such that each generator $[K,E]$ of $\CFmComb (G)$
 admits Maslov grading $\grad [K,E]$ at most $k_G$. Therefore
 $\HFmComb (G, \s )$ can be generated (over $\Field [U]$) by those
 elements which have Maslov grading between $q$ and $k_G$ (where these
 constants are chosen based on the discussion above). Since $G$ is
 negative definite, there are finitely many generators $U^j \otimes
 [K,E]$ of $\CFmComb (G, \s )$ with Maslov grading between $q$ and
 $k_G$, the claim of the corollary follows.  
\end{proof}

In preparing for the proof of Theorem~\ref{thm:fingenhat},
we verify a weaker version
of Lemma~\ref{lem:ehatok}, which now holds for any vertex.
\begin{lem}
  Suppose that $G$ is a negative definite plumbing graph and $v$ is
  one of its vertices. Then there are integers $m,n$ such that for a
  generator $[K,E]$ of $\CFmComb (G, \s )$
\begin{itemize}
\item $K(v)\geq n$ implies $a_v[K,E]=0$ and $b_v[K,E]>0$, and 
\item $K(v)\leq m$ implies $b_v[K,E]=0$ and $a_v[K,E]>0$.
\end{itemize}
\end{lem}
\begin{proof}
  The same argument as given for Lemma~\ref{lem:ehatok} applies: if we
  take $K(v) $ large enough then we get $A_v<B_v$ and hence the first
  claim follows, while for $K(v)$ negative enough we get
  $A_v>B_v$. Notice that since there are only finitely many subsets
  $E$, we can assume that the chosen $m,n$ depend only on $v$ and the
  conclusion of the lemma applies for all $[K,E]$.
\end{proof}

Fix a vertex $v$ and, similarly to $H_0$ of Section~\ref{sec:app},
define $H^v_0\colon \CFaComb (G, \s)\to \CFaComb (G, \s )$ as
\[
H_0^v[K,E]=
\left\{
\begin{array}{ll}
  0 & {\text{if $v\in E$ or $m< K(v)<n$}} \\
  \\
{[K,E\cup v]} 
  & {\text{if $v\not\in E$ and $K(v)\geq n$}} \\
  \\
 { [K-2v^*,E\cup v]} 
  & {\text{if $v\not\in E$ and $K(v)\leq m$}} \\
\end{array}
\right.
\]
Consider $C^v _0=Id +\partial \circ H_0^v+H^v _0\circ \partial$.  It
is easy to see that $C_0^v$ preserves the Maslov grading of an
element, and by its definition it is chain homotopic to the identity,
hence its image is a subcomplex which is chain homotopy equivalent to
the original chain complex.

\begin{defn}
  Let $B_v(N)$ denote the vector subspace of $\CFaComb (G, \s)$
  spanned by the generators $[K,E]$ which satisfy $\vert K(v)\vert
  \leq N$.  Let $B(N)$ denote the vector subspace generated by $[K,E]$
  with the property that $\vert K(v_i)\vert \leq N$ holds for all $v_i
  \in \Vertices (G)$.
\end{defn}

\begin{lem}\label{lem:hogynezki}
  For a given vertex $v$ there is an $N$ such that the image
  $C_0^v[K,E]$ is contained by $B_v(N)$.
\end{lem}
\begin{proof}
  We start by describing the image $C_0^v[K,E]$.  Assume first that
  $v\in E$. Then $H_0^v[K,E]=0$ and hence
  $C_0^v[K,E]=[K,E]+H_0^v\circ \partial [K,E]$. The second term is a
  sum of various terms originating from taking the boundary of $[K,E]$
  with respect to elements $w\in E$.  If $w\neq v$, then $v\in E-w$
  and hence these terms will be annihilated by $H_0^v$.  Therefore we
  need to examine only $H_0^v[K,E-v]$ (or $H^v_0[K+2v^*, E-v])$
  depending on whether $a_v[K,E]=0$ or $b_v[K,E]=0$ (or both)). For
  $K(v)\geq n$ or $K(v)\leq m$  the last term is equal to $[K,E]$,
  hence in this case $C_0^v[K,E]=0$. If $K(v)\in (m,n)$ then the last
  term is either 0 or (if $(K+2v^*)<m$) it is equal to $[K,E]$,
  implying that in this case the value $C_0^v[K,E]$ is either 0 or
  equals $[K,E]$.

  Suppose now that $v\not \in E$. Consider first the case when
  $K(v)\geq n$.  Then $(H_0^v \circ \partial + \partial \circ
  H_0^v)[K,E]$ has a number of terms (of the form $[K, E-w\cup v]$ and
  $[K+2w^*, E-w \cup v]$) which cancel each other, and the only
  remaining term will be equal to $[K,E]$. (In the cancellation we
  rely on the argument that the $U$-powers in front of the various
  terms should match up by Maslov grading reasons.) Adding this term
  to $[K,E]$ we conclude $C_0^v[K,E]=0$ is this case.  Suppose now
  that $K(v)\leq m$. This implies that $H_0^v[K,E]=[K-2v^*, E\cup v]$.
  The computation of $H_0^v\circ \partial$ on $[K,E]$ is slightly more
  complicated: if $K(v)\leq m-2$ then $K+2w^*$ still takes value $\leq
  m$ on $v$ (implying $C_0^v[K,E]=0$), but for $K(v)=m$ the value
  $(K+2w^*)(v)=m+2$, hence $H_0^v$ maps this term to 0, and so $C[K,E]$
  will have coordinates (besides $[K,E]$) of the form $[K-2v^*+2w^*,
    E-w\cup v]$. By choosing $N$ appropriately, these terms will be in
  $B_v(N)$. Finally, if $K(v)\in (m,n)$, then the claim $C_0^v[K,E]\in
  B_v(N)$ follows trivially from the definitions.
\end{proof}

\begin{proof}[Proof of Theorem~\ref{thm:fingenhat}]
  Fix an order $\{ v_1, \ldots , v_n\}$ on the vertex set $\Vertices
  (G)$ of $G$ and consider the map $C=C_0^{v_n}\circ \cdots \circ
  C_0^{v_1}$. All these maps are chain homotopic to the identity,
  hence the image of $C$ is a subcomplex of $\CFaComb (G, \s )$ which
  has homology isomorphic to $\HFaComb (G, \s )$.  By the repeated
  application Lemma~\ref{lem:hogynezki} it follows that there is $N$
  with the property that $C(\CFaComb (G, \s ))\subset B(N)$, and since
  $B(N)$ is obviously a finite dimensional vector space over $\Field$,
  the statement follows at once.
\end{proof}

\end{document}